\documentclass[A4paper,oneside,11pt]{article}
\usepackage{amssymb,amsthm,amsmath,marvosym,amsfonts,fontenc}
\usepackage[usenames,dvipsnames]{color}
\usepackage{enumerate,marginnote}
\usepackage[dvips]{graphicx}
\usepackage{mdwlist}
\usepackage{pifont}
\usepackage{bm}
\usepackage{multind}\makeindex{general}\makeindex{mathsymbols}
\usepackage{caption}
\usepackage{fullpage} 
\usepackage{epsfig,psfrag}
\usepackage{fancyhdr} 
\usepackage{xr}
\numberwithin{equation}{section}
\numberwithin{table}{section} 
\numberwithin{figure}{section}
\newtheorem{theorem}{Theorem}[section]
\newtheorem{subtheorem}{Theorem}[theorem]

\newtheorem{lemma}[theorem]{Lemma}
\newtheorem{informallemma}[theorem]{Informal Lemma}
\newtheorem{conjecture}[theorem]{Conjecture}

\newtheorem{proposition}[theorem]{Lemma}

\newtheorem{fact}[theorem]{Fact}
\newtheorem{remark}[theorem]{Remark}
\newtheorem{definition}[theorem]{Definition}

\theoremstyle{remark}
\newtheorem{claim}[subtheorem]{Claim}

\newcommand{\By}[2]{\overset{\mbox{\tiny{#1}}}{#2}}
\newcommand{\ByRef}[2]{   \By{\eqref{#1}}{#2} }
\newcommand{\eqBy}[1]{    \By{#1}{=} }
\newcommand{\lBy}[1]{     \By{#1}{<} }

\newcommand{\geBy}[1]{    \By{#1}{\ge} }
\newcommand{\eqByRef}[1]{ \ByRef{#1}{=} }

\newcommand{\leByRef}[1]{ \ByRef{#1}{\le} }

\let\sm\setminus
\let\subset\subseteq 
\let\supset\supseteq 
\let\epsilon\varepsilon
\let\sharp\#
\def\dcup{\dot\cup} 
\renewcommand{\leq}{\leqslant}
\renewcommand{\le}{\leqslant}
\renewcommand{\geq}{\geqslant}
\renewcommand{\ge}{\geqslant}

\let\oldmarginpar\marginpar
\renewcommand\marginpar[1]{\-\oldmarginpar[\raggedleft\footnotesize #1]%
{\raggedright\footnotesize #1}}

\newcommand{\HIDDENPROOF}[1]{
}

\newcommand{\HIDDENTEXT}[1]{
}
\title{The approximate Loebl--Koml\'os--S\'os Conjecture~II:\\ The rough structure of LKS graphs}
\author{Jan
Hladk\'y
\thanks{\emph{Corresponding author.} Institute of Mathematics, Academy of Science of the Czech Republic. \v Zitn\'a 25, 110 00, Praha, Czech Republic. The Institute of Mathematics of the Academy of Sciences of the Czech Republic is supported by RVO:67985840. Email:
\texttt{honzahladky@gmail.com}.
The research leading to these results has received funding from the People Programme (Marie Curie Actions) of the European Union's Seventh Framework Programme (FP7/2007-2013) under REA grant agreement umber 628974. Much of the work was done while supported by an EPSRC postdoctoral fellowship while affiliated with DIMAP and Mathematics Institute, University of
Warwick.}
\quad 
J\'anos Koml\'os\thanks{Department of Mathematics, Rutgers University, 110 Frelinghuysen Rd., Piscataway, NJ~08854-8019, USA} 
\quad 
Diana Piguet\thanks{Institute of Computer Science, Czech Academy of Sciences, Pod Vod\'arenskou v\v e\v z\'i 2, 182~07 Prague, Czech Republic. With institutional support RVO:67985807. Supported by the Marie Curie fellowship FIST, DFG grant TA 309/2-1,  Czech Ministry of Education project 1M0545, EPSRC award EP/D063191/1, and EPSRC Additional Sponsorship EP/J501414/1.
	The research leading to these results has received funding from the European Union Seventh
	Framework Programme (FP7/2007-2013) under grant agreement no. PIEF-GA-2009-253925.
    The work leading to this invention was supported by the European Regional Development Fund (ERDF), project ``NTIS -- New Technologies for Information Society'', European Centre of Excellence, CZ.1.05/1.1.00/02.0090.}
    \\ 
    Mikl\'os Simonovits\thanks{R\'enyi
    Institute, Budapest, Hungary. Supported by OTKA~78439, OTKA~101536, ERC-AdG.~321104} 
\quad 
Maya Stein\thanks{Department of Mathematical Engineering,
University of Chile, Santiago, Chile.  Supported by Fondecyt Iniciacion grant 11090141, Fondecyt Regular grant 1140766 and CMM Basal.}
\quad 
Endre Szemer\'edi\thanks{R\'enyi
	Institute, Budapest, Hungary. Supported by OTKA~104483 and ERC-AdG.~321104}}

\def\semiregular{regularized }
\def\semiregulars{regularized}
\def\Semiregular{Regularized }

\newcommand{\PARAMETERPASSING}[2]{{\mathrm{#1}\ref{#2}}}

\def\NN{\mathbb{N}}

\newcommand{\M}{\mathcal M}\newcommand{\C}{\mathcal C}\newcommand{\A}{\mathcal
A}\newcommand{\V}{\mathcal V}
\newcommand{\eps}{\epsilon}

\def\mindeg{\mathrm{mindeg}}
\def\maxdeg{\mathrm{maxdeg}}
\def\density{\mathrm{d}}
\def\neighbour{\mathrm{N}}

\newcommand{\treeclass}[1]{\mathbf{trees}({#1})}

\newcommand{\LKSgraphs}[3]{\mathbf{LKS}({#1},{#2},{#3})}

\newcommand{\LKSsmallgraphs}[3]{\mathbf{LKSsmall}({#1},{#2},{#3})}

\newcommand{\smallvertices}[3]{\mathbb{S}_{{#1},{#2}}({#3})}
\newcommand{\largevertices}[3]{\mathbb{L}_{{#1},{#2}}({#3})}

\newcommand{\JUSTIFY}[1]{\mbox{\tiny{(#1)}}\quad}

\def\Gcapt{G_\nabla}
\def\GD{G_{\mathcal{D}}}
\def\Gblack{G_{\mathrm{reg}}}
\def\Gexp{G_{\mathrm{exp}}}
\def\BGblack{\mathbf{G}_{\mathrm{reg}}}
\def\smallatoms{\mathbb{E}}
\def\clusters{\mathbf{V}}
\def\class{\nabla}
\def\HugeVertices{\mathbb{H}}
\def\DenseSpots{\mathcal{D}}

\def\shrubA{\mathcal S_{A}}
\def\shrubB{\mathcal S_{B}}

\def\XA{\mathbb{XA}}
\def\XB{\mathbb{XB}}
\def\XC{\mathbb{XC}}
\def\BS{\mathbf S} 
\def\BL{\mathbf L}
\def\BSN{\mathbf S^0} 
\def\BSI{\mathbf S^{\mathrm{I}}} 
\def\BSR{\mathbf S^\mathrm{R}}
\def\SEPARATOR{\mathbf Q} 
\def\SR{S^{\mathrm R}} 
\def\SN{S^{0}}
\def\Mgood{\M_{\mathrm{good}}}
\def\NAtom{{\mathcal N_{\smallatoms}}}

\def\clustersize{\mathfrak{c}}

\renewcommand{\today}{}
\date{}

\begin{document}
\pagenumbering{roman}
\maketitle
\begin{abstract}
This is the second of a series  of four papers in which
we prove the following relaxation of the
Loebl--Koml\'os--S\'os Conjecture: For every~$\alpha>0$
there exists a number~$k_0$ such that for every~$k>k_0$
 every $n$-vertex graph~$G$ with at least~$(\frac12+\alpha)n$ vertices
of degree at least~$(1+\alpha)k$ contains each tree $T$ of order~$k$ as a
subgraph. 

In the first paper of the series, we gave a
decomposition of the graph~$G$  into several parts of different characteristics; this decomposition might be viewed as an analogue of a regular partition for sparse graphs. 
In the present paper, we find a combinatorial structure inside this decomposition. In the last two papers, we refine the structure and use it for embedding the tree~$T$.
\end{abstract}


\newpage

\rhead{\today}

\tableofcontents
\newpage
\pagenumbering{arabic}
\setcounter{page}{1}

\section{Introduction}\label{sec:intro}

This is the second of a series of four papers~\cite{cite:LKS-cut0, cite:LKS-cut1, cite:LKS-cut2, cite:LKS-cut3} 
in which we provide an approximate solution of the Loebl--Koml\'os--S\'os Conjecture. The conjecture reads as follows.
 
\begin{conjecture}[Loebl--Koml\'os--S\'os Conjecture 1995~\cite{EFLS95}]\label{conj:LKS}
Suppose that $G$ is an $n$-vertex graph with at least $n/2$ vertices of degree more than $k-2$. Then $G$ contains each tree of order $k$.
\end{conjecture}

We discuss the history and state of the art in detail  in the first paper~\cite{cite:LKS-cut0} of our series. The main result, which will be proved in~\cite{cite:LKS-cut3}, is the 
approximate solution of the Loebl--Koml\'os--S\'os Conjecture.

\begin{theorem}[Main result~\cite{cite:LKS-cut3}]\label{thm:main}
For every $\alpha>0$ there exists a number $k_0$ such that for any
$k>k_0$ we have the following. Each $n$-vertex graph $G$ with at least
$(\frac12+\alpha)n$ vertices of degree at least $(1+\alpha)k$ contains each tree $T$ of
order $k$.
\end{theorem}

In the first paper~\cite{cite:LKS-cut0} we exposed the techniques we use to decompose the host graph. In particular, we saw in \cite[Lemma~\ref{p0.lem:LKSsparseClass}]{cite:LKS-cut0} that any graph satisfying the assumptions of Theorem~\ref{thm:main} may be decomposed into a set of huge degree vertices, regular pairs, an expanding subgraph, and another set with certain expansion properties, which we call the avoiding set. We call this a \emph{sparse decomposition} of a graph. We will recall the necessary notions from~\cite{cite:LKS-cut0} in Section~\ref{sec:FromPaper0}.

Many embedding problems for dense host graphs  are attacked using the following three-step approach: \emph{(a)} the regularity lemma is applied to the host graph, \emph{(b)} a suitable combinatorial structure is found in the cluster graph, and \emph{(c)} the target graph is embedded into the combinatorial structure using properties of regular pairs. If we consider the sparse decomposition as a sparse counterpart to~\emph{(a)} then the main result of the present paper, Lemma~\ref{prop:LKSstruct}, should be regarded as a counterpart to~\emph{(b)}. More precisely, for each graph satisfying the assertions of Theorem~\ref{thm:main} that is given together with its sparse decomposition, Lemma~\ref{prop:LKSstruct} gives a combinatorial structure whose building blocks are the elements of the sparse decomposition. As in tree embedding problems in the dense setting (e.g. in~\cite{AKS95,PS07+}), the core of this combinatorial structure is a well-connected matching consisting of regular pairs. We call such matchings \emph{\semiregulars}. 

With the structure given by Lemma~\ref{prop:LKSstruct}, one can convince oneself that the tree~$T$ from Theorem~\ref{thm:main} can be embedded into the host graph, and indeed we provide such motivation in Section~\ref{ssec:motivation}. However, the rigorous argument is far from trivial. One needs to refine the structure found here, which is done in~\cite{cite:LKS-cut2}. For this reason, we call the output of Lemma~\ref{prop:LKSstruct} {\it the rough structure}.
In the last paper~\cite{cite:LKS-cut3} of our series we will develop embedding techniques for trees, and finally prove Theorem~\ref{thm:main}.

\section{Notation and preliminaries}\label{sec:preliminaries}

\subsection{General notation}

The set $\{1,2,\ldots, n\}$ of the first $n$ positive integers is
denoted by \index{mathsymbols}{*@$[n]$}$[n]$. 
 We frequently employ indexing by many indices. We write
superscript indices in parentheses (such as $a^{(3)}$), as
opposed to notation of powers (such as $a^3$).
We use sometimes subscripts to refer to
parameters appearing in a fact/lemma/theorem. For example
$\alpha_\PARAMETERPASSING{T}{thm:main}$ refers to the parameter $\alpha$ from Theorem~\ref{thm:main}.
We omit rounding symbols when this does not affect the
correctness of the arguments.

Table~\ref{tab:notation} shows the system of notation we use in the series.
\begin{table}[h]
\centering
\caption{Specific notation used in the series.}
\label{tab:notation}
\begin{tabular}{r|l}
\hline
lower case Greek letters &  small positive constants ($\ll 1$)\\
                         & $\phi$ reserved for embedding; $\phi:V(T)\rightarrow V(G)$\\
\hline
upper case Greek letters & large positive constants ($\gg 1$)\\
\hline
one-letter bold& sets of clusters \\
\hline
bold (e.g., $\treeclass{k},\LKSgraphs{n}{k}{\eta}$)& classes of graphs\\
\hline
blackboard bold (e.g., $\HugeVertices,\smallatoms,\smallvertices{\eta}{k}{G},\XA$)& distinguished vertex sets except for\\
& $\NN$ which denotes the set $\{1,2,\ldots\}$\\
\hline 
script  (e.g., $\mathcal A,\mathcal D,\mathcal N$)& families (of vertex sets, ``dense spots'', and regular pairs)\\
\hline
$\class$(=nabla)&sparse decomposition (see Definition~\ref{sparseclassdef})\\
\hline
\end{tabular}
\end{table}

 We write \index{mathsymbols}{*VG@$V(G)$}$V(G)$ and \index{mathsymbols}{*EG@$E(G)$}$E(G)$ for the vertex set and edge set of a graph $G$, respectively. Further, \index{mathsymbols}{*VG@$v(G)$}$v(G)=|V(G)|$ is the order of $G$, and \index{mathsymbols}{*EG@$e(G)$}$e(G)=|E(G)|$ is its number of edges. If $X,Y\subset V(G)$ are two, not necessarily disjoint, sets of vertices we write \index{mathsymbols}{*EX@$e(X)$}$e(X)$ for the number of edges induced by $X$, and \index{mathsymbols}{*EXY@$e(X,Y)$}$e(X,Y)$ for the number of ordered pairs $(x,y)\in X\times Y$ such that $xy\in E(G)$. In particular, note that $2e(X)=e(X,X)$.

\index{mathsymbols}{*DEG@$\deg$}\index{mathsymbols}{*DEGmin@$\mindeg$}\index{mathsymbols}{*DEGmax@$\maxdeg$}
For a graph $G$, a vertex $v\in V(G)$ and a set $U\subset V(G)$, we write
$\deg(v)$ and $\deg(v,U)$ for the degree of $v$, and for the number of
neighbours of $v$ in $U$, respectively. We write $\mindeg(G)$ for the minimum
degree of $G$, $\mindeg(U):=\min\{\deg(u)\::\: u\in U\}$, and
$\mindeg(V_1,V_2)=\min\{\deg(u,V_2)\::\:u\in V_1\}$ for two sets $V_1,V_2\subset
V(G)$. Similar notation is used for the maximum degree, denoted by $\maxdeg(G)$.
The neighbourhood of a vertex $v$ is denoted by
\index{mathsymbols}{*N@$\neighbour(v)$}$\neighbour(v)$. We set $\neighbour(U):=\bigcup_{u\in
U}\neighbour(u)$. The symbol $-$ is used
for two graph operations: if $U\subset V(G)$ is a vertex
set then $G-U$ is the subgraph of $G$ induced by the set
$V(G)\setminus U$. If $H\subset G$ is a subgraph of $G$ then the graph
$G-H$ is defined on the vertex set $V(G)$ and corresponds
to deletion of edges of $H$ from $G$.
 Any graph with zero edges is called \index{general}{empty graph}\emph{empty}.
A family $\mathcal A$ of pairwise disjoint subsets of $V(G)$ is an \index{general}{ensemble}\index{mathsymbols}{*ENSEMBLE@$\ell$-ensemble}\emph{$\ell$-ensemble in $G$} if  $|A|\ge \ell$ for each $A\in\mathcal A$. 

Finally, \index{mathsymbols}{*trees@$\treeclass{k}$}$\treeclass{k}$ denotes the class of all trees of order $k$.

\subsection{Regular pairs}

Given a graph $H$ and a pair $(U,W)$ of disjoint
sets $U,W\subset V(H)$ the
\index{general}{density}\index{mathsymbols}{*D@$\density(U,W)$}\emph{density of the pair $(U,W)$} is defined as
$$\density(U,W):=\frac{e(U,W)}{|U||W|}\;.$$
For a given $\varepsilon>0$, a pair $(U,W)$ of disjoint
sets $U,W\subset V(H)$ 
is called an \index{general}{regular pair}\emph{$\epsilon$-regular
pair} if $|\density(U,W)-\density(U',W')|<\epsilon$ for every
$U'\subset U$, $W'\subset W$ with $|U'|\ge \epsilon |U|$, $|W'|\ge
\epsilon |W|$. If the pair $(U,W)$ is not $\epsilon$-regular,
then we call it \index{general}{irregular}\emph{$\epsilon$-irregular}.

We shall need a useful and well-known property of
regular pairs.
\begin{fact}\label{fact:BigSubpairsInRegularPairs}
Suppose that $(U,W)$ is an $\varepsilon$-regular pair of density
$d$. Let $U'\subset W, W'\subset W$ be sets of vertices with $|U'|\ge
\alpha|U|$, $|W'|\ge \alpha|W|$, where $\alpha>\epsilon$.
Then the pair $(U',W')$ is a $2\varepsilon/\alpha$-regular pair of density at least
$d-\varepsilon$.
\end{fact}

The regularity lemma~\cite{Sze78} has proved to be
a powerful tool for attacking graph embedding
problems; see~\cite{KuhnOsthusSurv} 
for a survey. 

\begin{lemma}[Regularity lemma]\label{lem:RL}
For all $\epsilon>0$ and $\ell\in\NN$ there exist $n_0,M\in\NN$
such that for every $n\ge n_0$ the following holds. Let $G$ be an $n$-vertex graph whose
vertex set is pre-partitioned into sets
$V_1,\ldots,V_{\ell'}$, $\ell'\le \ell$. Then there exists
a partition $U_0,U_1,\ldots,U_p$ of $V(G)$, $\ell<p<M$, with the following properties.
\begin{enumerate}[(1)]
\item For every $i,j\in [p]$ we have $|U_i|=|U_j|$, and  $|U_0|<\epsilon n$.
\item For every
$i\in [p]$ and every $j\in [\ell']$ either $U_i\cap
V_j=\emptyset$ or $U_i\subset
V_j$. 
\item All but at most
$\epsilon p^2$ pairs $(U_i,U_j)$, $i,j\in [p]$, $i\neq j$, are $\epsilon$-regular.
\end{enumerate}
\end{lemma}

We shall use Lemma~\ref{lem:RL}  for auxiliary
purposes only as it is helpful only in the setting of dense
graphs (i.e., graphs which have $\ell$ vertices and
$\Omega(\ell^2)$ edges).

\subsection{LKS graphs}\index{mathsymbols}{*LKSgraphs@$\LKSgraphs{n}{k}{\eta}$}

It will be convenient to restrict our attention to a class of graphs which is in a way minimal for Theorem~\ref{thm:main}. Write 
\index{mathsymbols}{*LKSgraphs@$\LKSgraphs{n}{k}{\eta}$}$\LKSgraphs{n}{k}{\alpha}$
for the class of all $n$-vertex graphs with at least
$(\frac12+\alpha)n$ vertices of degrees at least
$(1+\alpha)k$. With this notation Conjecture~\ref{conj:LKS} states that every graph in $\LKSgraphs{n}{k}{0}$ contains every tree from $\treeclass{k+1}$.

Given a graph $G$, denote by
\index{mathsymbols}{*S@$\smallvertices{\eta}{k}{G}$}$\smallvertices{\eta}{k}{G}$ the set of those
vertices of $G$ that have degree less than $(1+\eta)k$ and by
\index{mathsymbols}{*L@$\largevertices{\eta}{k}{G}$}$\largevertices{\eta}{k}{G}$ the set of those
vertices of $G$ that have degree at least $(1+\eta)k$.
When proving Theorem~\ref{thm:main}, we may of course  restrict our attention to LKS-minimal graphs, that is, to graphs that are edge-minimal with respect to belonging to $\LKSgraphs{n}{k}{\alpha}$. It is easy to show that in each such graph the set $\smallvertices{\eta}{k}{G}$ is independent,  all the neighbours of every vertex $v\in V(G)$ with $\deg(v)>\lceil(1+\eta)k\rceil$ have degree exactly $\lceil(1+\eta)k\rceil$, and $|\largevertices{\eta}{k}{G}|\le\lceil
(1/2+\eta)n\rceil+1$. It turns out that our main decomposition result \cite[Lemma~\ref{p0.lem:LKSsparseClass}]{cite:LKS-cut0}  outputs a graph with slightly weaker properties than being LKS-minimal. Let us therefore introduce the following class of graphs.
\begin{definition}\label{def:LKSsmall}
Suppose that $n,k\in\mathbb N$, and $\eta>0$.
Let \index{mathsymbols}{*LKSsmallgraphs@$\LKSsmallgraphs{n}{k}{\eta}$}$\LKSsmallgraphs{n}{k}{\eta}$ be the class of those graphs $G\in\LKSgraphs{n}{k}{\eta}$ for which we have the following three properties:
\begin{enumerate}[(i)]
   \item\label{en:LKSsmall.def:LKSsmall} All the neighbours of every vertex $v\in V(G)$ with $\deg(v)>\lceil(1+2\eta)k\rceil$ have degrees at most $\lceil(1+2\eta)k\rceil$.\label{def:LKSsmallA}
   \item\label{en:LKSsmall.noSS} All the neighbours of every vertex of $\smallvertices{\eta}{k}{G}$
    have degree exactly $\lceil(1+\eta)k\rceil$. \label{def:LKSsmallB}
   \item We have $e(G)\le kn$.\label{def:LKSsmallC}
\end{enumerate}
\end{definition}

\section{Decomposing sparse graphs}\label{sec:FromPaper0}
In~\cite{cite:LKS-cut0} we introduced the notion of sparse decomposition, and proved that every graph can be (almost perfectly) decomposed. We define the sparse decomposition after introducing its basic building blocks: dense spots and avoiding sets. For motivation and more details we refer the reader to~\cite[Section~\ref{p0.ssec:class-black}]{cite:LKS-cut0}, of which  this section is a condensed version.

We start by defining dense spots. These are bipartite graphs having positive density, and will (among other things) serve as a basis for regularization.
\begin{definition}[\bf \index{general}{dense spot}$(m,\gamma)$-dense spot,
\index{general}{nowhere-dense}$(m,\gamma)$-nowhere-dense]\label{def:densespot} 
Suppose that $m\in\NN$ and $\gamma>0$.
An \emph{$(m,\gamma)$-dense spot} in a graph $G$ is a non-empty bipartite sub\-graph  $D=(U,W;F)$ of  $G$ with
$\density(D)>\gamma$ and $\mindeg (D)>m$. We call a graph $G$
\emph{$(m,\gamma)$-nowhere-dense} if it does not contain any $(m,\gamma)$-dense spot.

When the parameters $m$ and $\gamma$ are irrelevant, we refer to $D$ simply as a \emph{dense spot}.
\end{definition}
Note that dense spots do not have any 
specified orientation. That is, we view $(U,W;F)$ and $(W,U;F)$ as
the same object.

\begin{definition}[\bf $(m,\gamma)$-dense
cover]\index{general}{dense cover}
Suppose that $m\in\NN$ and $\gamma>0$.
 An \emph{$(m,\gamma)$-dense cover} of a given
graph $G$ is a family $\DenseSpots$ of edge-disjoint
$(m,\gamma)$-dense
spots such that $E(G)=\bigcup_{D\in\DenseSpots}E(D)$.
\end{definition}

We now define the avoiding set. Informally, a set $\smallatoms$ of vertices is avoiding if for each set $U$ of size up to $\Lambda k$ (where $\Lambda\gg 1$ is a large constant) and each vertex $v\in\smallatoms$ there is a dense spot containing $v$ and almost disjoint from $U$. Favourable properties of avoiding sets for embedding trees are shown in~\cite[Section~\ref{p0.sssec:whyavoiding}]{cite:LKS-cut0}.
\begin{definition}[\bf
\index{general}{avoiding}$(\Lambda,\epsilon,\gamma,k)$-avoiding set]\label{def:avoiding} 
Suppose that $k\in\NN$, $\epsilon,\gamma>0$ and $\Lambda>0$. Suppose that~$G$ is a graph and $\DenseSpots$ is a family of dense spots in $G$. A set
$\smallatoms\subset \bigcup_{D\in\DenseSpots} V(D)$ is \emph{$(\Lambda,\epsilon,\gamma,k)$-avoiding} with
respect to $\DenseSpots$ if for every $U\subset V(G)$ with $|U|\le \Lambda k$ the following holds for all but at most $\epsilon k$ vertices $v\in\smallatoms$. There is a dense spot $D\in\DenseSpots$ with $|U\cap V(D)|\le \gamma^2 k$ that contains $v$.
\end{definition}

We can now introduce an auxiliary notion of bounded decomposition on which we can build the key concept of sparse decomposition (see below). The main result in~\cite{cite:LKS-cut0} tells us that every graph has an almost perfect sparse decomposition. This sparse decomposition (and the bounded decomposition included in it) will provide us with control on the behaviour of the different edge and vertex sets involved, and thus be helpful to embed the tree.

\begin{definition}[\index{general}{bounded decomposition}{\bf
$(k,\Lambda,\gamma,\epsilon,\nu,\rho)$-bounded decomposition}]\label{bclassdef}
Suppose that $k\in\NN$ and $\epsilon,\gamma,\nu,\rho>0$ and $\Lambda>0$. 
Let $\mathcal V=\{V_1, V_2,\ldots, V_s\}$ be a partition of the vertex set of a graph $G$. We say that $( \clusters,\DenseSpots, \Gblack, \Gexp,
\smallatoms )$ is a {\em $(k,\Lambda,\gamma,\epsilon,\nu,\rho)$-bounded
decomposition} of $G$ with respect to $\mathcal V$ if the following properties
are satisfied:
\begin{enumerate}
\item\label{defBC:nowheredense}
$\Gexp$ is  a $(\gamma k,\gamma)$-nowhere-dense subgraph of $G$ with $\mindeg(\Gexp)>\rho k$.
\item\label{defBC:clusters} $\clusters$ is a family of  disjoint subsets of 
$ V(G)$.
\item\label{defBC:RL} $\Gblack$ is a subgraph of $G-\Gexp$ on the vertex set $\bigcup \clusters$. For each edge
 $xy\in E(\Gblack)$ there are distinct $C_x\ni x$ and $C_y\ni y$ from $\clusters$,
and  $G[C_x,C_y]=\Gblack[C_x,C_y]$. Furthermore, 
$G[C_x,C_y]$ forms an $\epsilon$-regular pair of  density at least $\gamma^2$.
\item We have $\nu k\le |C|=|C'|\le \epsilon k$ for all
$C,C'\in\clusters$.\label{Csize}
\item\label{defBC:densepairs}  $\DenseSpots$ is a family of edge-disjoint $(\gamma
k,\gamma)$-dense spots  in $G-\Gexp$.  For
each $D=(U,W;F)\in\DenseSpots$ all the edges of $G[U,W]$ are covered
by $\DenseSpots$ (but not necessarily by $D$).
\item\label{defBC:dveapul} If  $\Gblack$
contains at least one edge between $C_1,C_2\in\clusters$ then there exists a dense
spot $D=(U,W;F)\in\DenseSpots$ such that $C_1\subset U$ and $C_2\subset
W$.
\item\label{defBC:prepartition}
For
all $C\in\clusters$ there is a set $V\in\mathcal V$ so that either $C\subseteq V\cap V(\Gexp)$ or $C\subseteq V\setminus V(\Gexp)$.
For
all $C\in\clusters$ and $D=(U,W; F)\in\DenseSpots$ we have $C\cap U,C\cap W\in\{\emptyset, C\}$.
\item\label{defBC:avoiding}
$\smallatoms$ is a $(\Lambda,\epsilon,\gamma,k)$-avoiding subset  of
$V(G)\setminus \bigcup \clusters$ with respect to the family of dense spots $\DenseSpots$.
\end{enumerate}

\smallskip
We say that the bounded decomposition $(\clusters,\DenseSpots, \Gblack, \Gexp,
\smallatoms )$ {\em respects the avoiding threshold~$b$}\index{general}{avoiding threshold} if for each $C\in \clusters$ we either have $\maxdeg_G(C,\smallatoms)\le b$, or $\mindeg_G(C,\smallatoms)> b$.
\end{definition}

The members of $\clusters$ are called \index{general}{cluster}{\it clusters}. Define the
{\it cluster graph} \index{mathsymbols}{*Gblack@$\BGblack$}  $\BGblack$ as the graph
on the vertex set $\clusters$ that has an edge $C_1C_2$
for each pair $(C_1,C_2)$ which has density at least $\gamma^2$ in the graph
$\Gblack$. Further, we define the graph \index{mathsymbols}{*GD@$\GD$}$\GD$ as the union (both edge-wise, and vertex-wise) of all dense spots~$\DenseSpots$. 
%
%

We now
enhance the structure of bounded decomposition by adding one new feature: vertices of very large degree.

\begin{definition}[\bf \index{general}{sparse
decomposition}$(k,\Omega^{**},\Omega^*,\Lambda,\gamma,\epsilon,\nu,\rho)$-sparse decomposition]\label{sparseclassdef}
Suppose that $k\in\NN$ and $\epsilon,\gamma,\nu,\rho>0$ and $\Lambda,\Omega^*,\Omega^{**}>2$. 
Let $\mathcal V=\{V_1, V_2,\ldots, V_s\}$ be a partition of the vertex set of a graph $G$. We say that 
$\class=(\HugeVertices, \clusters,\DenseSpots, \Gblack, \Gexp, \smallatoms )$
is a 
{\em $(k,\Omega^{**},\Omega^*,\Lambda,\gamma,\epsilon,\nu,\rho)$-sparse decomposition} of $G$
with respect to $V_1, V_2,\ldots, V_s$ if the following hold.
\begin{enumerate}
\item\label{def:classgap} $\HugeVertices\subset V(G)$,
$\mindeg_G(\HugeVertices)\ge\Omega^{**}k$,
$\maxdeg_H(V(G)\setminus \HugeVertices)\le\Omega^{*}k$, where $H$ is spanned by the edges of $\bigcup\DenseSpots$, $\Gexp$, and
edges incident with $\HugeVertices$,
\item \label{def:spaclahastobeboucla} $( \clusters,\DenseSpots, \Gblack,\Gexp,\smallatoms)$ is a 
$(k,\Lambda,\gamma,\epsilon,\nu,\rho)$-bounded decomposition of
$G-\HugeVertices$ with respect to $V_1\setminus \HugeVertices, V_2\setminus \HugeVertices,\ldots, V_s\setminus \HugeVertices$.
\end{enumerate}
\end{definition}

If the parameters do not matter, we call $\class$ simply a {\em sparse
decomposition}, and similarly we speak about a {\em bounded decomposition}. 

 \begin{definition}[\bf \index{general}{captured edges}captured edges]\label{capturededgesdef}
In the situation of Definition~\ref{sparseclassdef}, we refer to the edges in
$ E(\Gblack)\cup E(\Gexp)\cup
E_G(\HugeVertices,V(G))\cup E_G(\smallatoms,\smallatoms\cup \bigcup \clusters)$
as \index{general}{captured edges}{\em captured} by the sparse decomposition. 
 We write
\index{mathsymbols}{*Gclass@$\Gcapt$}$\Gcapt$ for the subgraph of $G$ on the vertex set $V(G)$ which consists of the captured edges.
Likewise, the captured edges of a bounded decomposition
$(\clusters,\DenseSpots, \Gblack,\Gexp,\smallatoms )$ of a graph $G$ are those
in $E(\Gblack)\cup E(\Gexp)\cup E_{\GD}(\smallatoms,\smallatoms\cup\bigcup\clusters)$.
\end{definition}

It will be useful to have the following shorthand notation at hand.

\begin{definition}[$\mathcal G(n,k,\Omega,\rho,\nu, \tau)$ and $\bar{\mathcal G}(n,k,\Omega,\rho,\nu)$]\label{tupelclass}
Suppose that $k,n\in\NN$ and $\nu,\rho,\tau>0$ and $\Omega>0$. 
We define $\mathcal G(n,k,\Omega,\rho,\nu, \tau)$\index{mathsymbols}{*G@$\mathcal G(n,k,\Omega,\rho,\nu, \tau)$} to be the class of all quadruple $(G,\DenseSpots,H,\mathcal A)$ with the following properties:
\begin{enumerate}[(i)]
\item  $G$ is a graph of order $n$ with $\maxdeg(G)\le \Omega k$,\label{maxroach}
\item $H$ is a bipartite subgraph of
$G$ with colour classes $A_H$ and $B_H$ and with $e(H)\ge \tau kn$,\label{duke}
\item  $\DenseSpots$ is a $(\rho k, \rho)$-dense cover of $G$,
\item $\mathcal A$ is a $(\nu k)$-ensemble in $G$,
and $A_H\subseteq \bigcup \mathcal A$,\label{bird}
\item  $A\cap U,A\cap W\in\{\emptyset,A\}$ for each $A\in\mathcal A$ and for each $D=(U,W;F)\in\DenseSpots$.\label{mingus}
\end{enumerate}
Those $G$, $\mathcal D$ and $\mathcal A$ satisfying all conditions but~\eqref{duke} and the last part of~\eqref{bird} will make up the triples $(G,\DenseSpots,\mathcal A)$ of the class  $\bar{\mathcal G}(n,k,\Omega,\rho,\nu)$\index{mathsymbols}{*G@$\bar{\mathcal G}(n,k,\Omega,\rho,\nu)$}.
\end{definition}

\section{Augmenting a matching}\label{sec:augmenting}
In previous papers~\cite{AKS95,Z07+,PS07+,Cooley08,HlaPig:LKSdenseExact}
concerning the LKS~Conjecture in the dense setting the crucial turn was to find
a matching in the cluster graph of the host graph possessing certain properties.
We will prove a similar ``structural result'' in Section~\ref{sec:LKSStructure}.
In the present section, we prove the main tool for
Section~\ref{sec:LKSStructure}, namely Lemma~\ref{lem:Separate}. All statements preceding Lemma~\ref{lem:Separate} are only preparatory. The only exception is (the easy) Lemma~\ref{lem:edgesEmanatingFromDensePairsIII} which is recycled later, in~\cite{cite:LKS-cut2}.

\subsection{\Semiregular matchings}

We prove our first auxiliary lemma on our way towards
Lemma~\ref{lem:Separate}.

\begin{lemma}\label{lem:edgesEmanatingFromDensePairsII}
For every $\Omega\in\NN$ and $\epsilon,\rho, \tau>0$ 
there is a number $\alpha>0$ such that for every $\nu\in(0,1)$
there exists a number $k_0\in\NN$ such
that for each $k>k_0$ the following holds. 

For every $(G,\DenseSpots,H,\mathcal A)\in\mathcal
G(n,k,\Omega,\rho,\nu,\tau)$ there are
$(U,W;F)\in\DenseSpots$, $A\in\mathcal A$
and $X,Y\subseteq V(G)$ such that 
\begin{enumerate}[(1)]
  \item $|X|=|Y|\ge\alpha\nu k$,
  \item  $X\subset A\cap U\cap A_H$ and $Y\subset W\cap B_H$, where $A_H$ and
  $B_H$ are the colour classes of $H$, and
  \item $(X,Y)$ is an $\epsilon$-regular pair in $G$ of density
  $\density(X,Y)\ge \frac{\tau\rho}{4\Omega}$.
\end{enumerate}
\end{lemma}

\begin{proof}
Let $\Omega$, $\epsilon$, $\rho$ and $\tau$ be given. Applying
Lemma~\ref{lem:RL} to
$\epsilon_\PARAMETERPASSING{L}{lem:RL}:=\min\{\epsilon
,\frac{\rho^2}{8\Omega}\}$
and $\ell_\PARAMETERPASSING{L}{lem:RL}:=2$, we obtain numbers $n_0$ and $M$.
We set 
\begin{equation}\label{allllpha}
\alpha:=\frac{\tau\rho}{\Omega^2M},
\end{equation}
 and given $\nu\in(0,1)$, we set
$$k_0:=\frac{2n_0}{\alpha\nu M}.$$
Now suppose we are given  $k>k_0$ and $(G,\DenseSpots,H,\mathcal A)\in\mathcal G(n,k,\Omega,\rho,\nu,\tau)$.

Property~\eqref{maxroach} of Definition~\ref{tupelclass} gives that $e(G)\le \Omega
k n/2$, and Property~\eqref{duke} says  that $e(H)\ge
\tau kn$. So $e(H)/e(G)\geq 2\tau/ \Omega$. Averaging over all dense spots $\DenseSpots$ in the dense cover of $G$ we find a dense spot $D=(U,W;F)\in\DenseSpots$ such that
\begin{align}\label{eq:proporcialniDensity}
e_{D}(A_H,B_H)= |F\cap E(H)|\geq \frac{e(H)}{e(G)}|F|\ge
\frac{2\tau |F|}{\Omega}\;.
 \end{align}

Without loss of generality, we assume that
\begin{equation}\label{eq:WLOGUW}
e_{D}(U\cap A_H,W\cap B_H)\geq \frac12\cdot e_{D}(A_H,B_H) \ge e_{D}(U\cap B_H,W\cap A_H)\;,
\end{equation}
as otherwise one can just interchange the roles of $U$ and $W$.
Then,
\begin{align}
e_G(U\cap A_H,W\cap B_H)&\geBy{\eqref{eq:WLOGUW}} \frac12\cdot e_{D}(A_H,B_H) \geBy{\eqref{eq:proporcialniDensity}}\frac{\tau}{\Omega}\cdot |F|.\label{colchon}
\end{align}

Let $\mathcal A'\subset\mathcal A$ denote the family of
those $A\in\mathcal A$ with
$0< e_G(A\cap U\cap
A_H,W\cap B_H)< \frac{\tau}{\Omega}\cdot |F|\cdot \frac{|A|}{|U|}$. 
Note that for each $A\in\A'$ we have $A\subset U$ by Definition~\ref{tupelclass}~\eqref{mingus}. Therefore,
\begin{equation*}
e_G\left(\bigcup \mathcal A'\cap U\cap A_H,W\cap
B_H\right)
< \ \frac \tau\Omega \cdot|F|\cdot \frac{|\A'|}{|U|}
\leq \ \frac \tau\Omega \cdot|F|
\overset{\eqref{colchon}}\leq \ e_G(U\cap A_H,W\cap B_H)\;.
\end{equation*}

As $\mathcal A$ covers $A_H$, $G$ has an edge $xy$ with $x\in U\cap A_H\cap A$ for some $A\in\mathcal A\setminus \mathcal A'$ and $y\in W\cap
B_H$.  Set $X':=A\cap U\cap A_H=A\cap A_H$ and $Y':=W\cap B_H$. Then directly
from the definition of $\mathcal A'$ and since $D$ is a $(\rho k,\rho)$-dense spot, 
we obtain that
\begin{equation}
\label{eq:denX'Y'}
\density_G(X',Y')= \frac{e_G(X',Y')}{|X'||Y'|}  \ge
\frac{\frac{\tau}{\Omega}\cdot |F|\cdot \frac{|A|}{|U|}}{|A||W|}  >
\ \frac{\tau\rho}{\Omega}.
\end{equation}

Also,
since $(U,W;F)\in\DenseSpots$, we have
\begin{equation}\label{eq:denseIndMany}
|F|\ge \rho k |U|\;.
\end{equation}
This enables us to bound the size of $X'$ as follows.
\begin{align}
\begin{split}
\label{eq:sizPreX'}
|X'| &\geq
\frac{e_G(X',Y')}{\maxdeg{(G)}}
 \\[6pt]
\JUSTIFY{as $A\not\in\mathcal A'$ and by D\ref{tupelclass}\eqref{maxroach}}&  \ge
 \frac{\frac{\tau}{\Omega}\cdot \frac{|F|}{|U|}\cdot |A|}{\Omega k} \\[6pt]
\JUSTIFY{by \eqref{eq:denseIndMany}}& 
\ge \frac{\tau\cdot \rho k\cdot |A|}{\Omega^2 k} \\
&\ge \frac{\tau\rho \nu k}{\Omega^2 }\\
&\eqByRef{allllpha} \alpha \nu kM\;.
\end{split}
\end{align}

Similarly,
\begin{align}
\label{eq:sizPreY'}
|Y'|
 \geq \ \alpha\nu kM\;.
\end{align}

Applying Lemma~\ref{lem:RL} to $G[X',Y']$ with prepartition
$\{X',Y'\}$ we obtain a collection of sets $\mathcal
C=\{C_i\}_{i=0}^p$, with $p<M$. By~\eqref{eq:sizPreX'},
and~\eqref{eq:sizPreY'}, we have that $|C_i|\ge \alpha\nu k$
for every $i\in[p]$. It is easy to
deduce from~\eqref{eq:denX'Y'} that
there is at least one $\epsilon_\PARAMETERPASSING{L}{lem:RL}$-regular (and thus $\epsilon$-regular) pair $(X,Y)$, $X,Y\in\mathcal
C\setminus\{C_0\}$, $X\subset X'$, $Y\subset Y'$ with
$\density(X,Y)\ge \frac{\tau\rho}{4\Omega}$. Indeed, it suffices to count the number of edges incident with $C_0$, lying in $\epsilon_{\mathrm L\ref{lem:RL}}$-irregular pairs or
belonging to too sparse pairs. The number of these ``bad'' edges is  strictly smaller than
$$(\epsilon_\PARAMETERPASSING{L}{lem:RL} +
\epsilon_\PARAMETERPASSING{L}{lem:RL} + \frac{\rho^2}{4\Omega}
)|X'||Y'| \leq\frac{\rho^2}{2\Omega} |X'||Y'|
\overset{\eqref{eq:denX'Y'}}\leq e(X',Y').$$ 
Thus not
all edges between $X'$ and $Y'$ are bad in the sense above. This finishes the
proof of Lemma~\ref{lem:edgesEmanatingFromDensePairsII}.
\end{proof}

Instead of just one pair $(X,Y)$, as it is given by 
Lemma~\ref{lem:edgesEmanatingFromDensePairsII}, we shall later need several disjoint pairs for embedding larger trees. For this purpose we introduce the following definition, generalizing the notion of a matching in the cluster graph in the traditional regularity setting.

\begin{definition}[\bf $(\epsilon,d,\ell)$-\semiregular matching]\label{def:semiregular} 
Suppose that $\ell\in\NN$ and  $d,\epsilon>0$.
A collection~$\mathcal N$ of ordered pairs $(A,B)$ with $A,B\subset V(H)$ is called an
\index{general}{regularized@\semiregular matching}\emph{$(\epsilon,d,\ell)$-\semiregular matching} of a
graph~$H$ if \begin{enumerate}[(i)] 
\item\label{def:semi1} $|A|=|B|\ge \ell$ for each
$(A,B)\in\mathcal N$, \item $(A,B)$ induces in $H$ an $\epsilon$-regular pair of
density at least $d$, for each $(A,B)\in\mathcal N$, and \item all involved sets
$A$ and $B$ are  pairwise disjoint.
\end{enumerate}
Sometimes, when the parameters do not matter (as for instance in Definition~\ref{altPath} below) we simply call it a \emph{\semiregular matching}.
\end{definition}

For a \semiregular matching $\mathcal N$, we shall write 
\index{mathsymbols}{*V1@$\V_1(\M)$, $\V_2(\M)$, $\V(\M)$}$\V_1(\mathcal
N):=\{A\::\:(A,B)\in\mathcal N\}$, $\V_2(\mathcal N):=\{B\::\:(A,B)\in\mathcal N\}$
and $\V(\mathcal N):=\mathcal V_1(\mathcal N)\cup \mathcal V_2(\mathcal N)$. 
Furthermore, we set 
\index{mathsymbols}{*V1@$V_1(\mathcal M)$,
$V_2(\mathcal M)$, $V(\mathcal M)$} 
$V_1(\mathcal N):=\bigcup \mathcal V_1(\mathcal N)$,
$V_2(\mathcal N):=\bigcup \mathcal V_2(\mathcal N)$ and $V(\mathcal N):=V_1(\mathcal
N)\cup V_2(\mathcal N)= \bigcup \mathcal V(\mathcal N)$. 
 As these definitions suggest, the orientations of
the pairs $(A,B)\in\mathcal N$ are important. The sets $A$ and $B$ are called \index{mathsymbols}{*VERTEX@$\M$-vertex}\index{general}{vertex@$\M$-vertex}\emph{$\mathcal N$-vertices} and the pair $(A,B)$ is an \index{mathsymbols}{*EDGE@$\M$-edge}\index{general}{edge@$\M$-edge}\emph{$\mathcal N$-edge}.

We say that a \semiregular matching $\mathcal N$
\index{general}{absorb}\emph{absorbs} a \semiregular matching $\mathcal M$ if for every $(S,T)\in\mathcal M$ there exists $(X,Y)\in\mathcal N$ such that $S\subset X$
and $T\subset Y$. In the same way, we say that a family of dense spots $\mathcal D$
\index{general}{absorb}\emph{absorbs} a \semiregular matching $\mathcal M$ if for every $(S,T)\in\mathcal M$ there exists $(U,W;F)\in\mathcal D$ such that $S\subset U$
and $T\subset W$.

We  later need
the following easy bound on the size of the elements  of
$\mathcal V(\mathcal M)$.

\begin{fact}\label{fact:boundMatchingClusters}
Suppose that $\mathcal M$ is an
$(\epsilon,d,\ell)$-\semiregular
matching in a graph $H$. Then $|C|\le
\frac{\maxdeg{(H)}}{d}$ for each $C\in
\mathcal V(\mathcal M)$.
\end{fact}
\begin{proof}
Let for example $(C,D)\in\mathcal M$. The maximum degree of $H$ is at least as large as the average
degree of the vertices in $D$, which is at least
$d |C|$.
\end{proof}

The  second step towards Lemma~\ref{lem:Separate} is 
Lemma~\ref{lem:edgesEmanatingFromDensePairsIII}. Whereas
Lemma~\ref{lem:edgesEmanatingFromDensePairsII} gives
one dense regular pair, in the same setting 
Lemma~\ref{lem:edgesEmanatingFromDensePairsIII} provides us with
a dense \semiregular
matching. 

\begin{lemma}\label{lem:edgesEmanatingFromDensePairsIII}
For every $\Omega\in\NN$ and $\rho,\epsilon,\tau \in(0,1)$ there exists
$\alpha>0$ such that for every $\nu\in (0,1)$ there is a number $k_0\in\NN$
such that the following holds for every $k>k_0$. 

 For each $(G,\DenseSpots,H,\mathcal A)\in\mathcal G(n,k,\Omega,\rho,\nu,\tau)$ there exists an
$(\epsilon ,\frac{\tau\rho}{8\Omega},\alpha\nu k)$-\semiregular matching $\mathcal M$ of $G$ such that
\begin{enumerate}
  \item[$\mathbf{(P1)}$] for each $(X,Y)\in\mathcal M$ there are $A\in\mathcal A$, and
  $D=(U,W;F)\in \DenseSpots$ such that  $X\subset U\cap A\cap A_H$ and $Y\subset
  W\cap B_H$,
  and
  \item[$\mathbf{(P2)}$] $|V(\mathcal M)|\ge\frac{\tau}{2\Omega} n$.
\end{enumerate}
\end{lemma}

\begin{proof}
Let
$\alpha:=\alpha_\PARAMETERPASSING{L}{lem:edgesEmanatingFromDensePairsII}>0$ be given by Lemma~\ref{lem:edgesEmanatingFromDensePairsII}
  for the  input parameters
  $\Omega_\PARAMETERPASSING{L}{lem:edgesEmanatingFromDensePairsII}:=\Omega$,
  $\epsilon_\PARAMETERPASSING{L}{lem:edgesEmanatingFromDensePairsII}:=\epsilon$,
  $\tau_\PARAMETERPASSING{L}{lem:edgesEmanatingFromDensePairsII}:=\tau/2$
  and $\rho_\PARAMETERPASSING{L}{lem:edgesEmanatingFromDensePairsII}:=\rho$.
  For $\nu_\PARAMETERPASSING{L}{lem:edgesEmanatingFromDensePairsII}:=\nu$,
  Lemma~\ref{lem:edgesEmanatingFromDensePairsII}  yields a number $k_0\in\NN$.

Now let  $(G,\DenseSpots,H,\mathcal A)\in\mathcal G(n,k,\Omega,\rho,\nu,\tau)$. Let $\mathcal M$ be an inclusion-maximal
$(\epsilon,\frac{\tau\rho}{8\Omega},\alpha\nu k)$-\semiregular matching with property~$\mathbf{(P1)}$. 
We claim that
\begin{equation}\label{bedingung2}
  e_G(A_H\setminus V_1(\mathcal M),B_H\setminus V_2(\mathcal M))<\frac\tau 2 kn.
\end{equation}
Indeed, suppose the contrary. Then 
the bipartite subgraph $H'$ of $G$ induced by the sets $A_H\setminus
V_1(\mathcal M)=A_H\setminus V(\mathcal M)$ and $B_H\setminus V_2(\mathcal M)=B_H\setminus V(\mathcal M)$
satisfies Property~\eqref{duke} of Definition~\ref{tupelclass}, with
$\tau_\PARAMETERPASSING{D}{tupelclass}:=\tau/2$. So, we have that $(G, \DenseSpots, H', \mathcal A)\in \mathcal G(n,k,\Omega,\rho,\nu,\tau/2)$.

Thus
 Lemma~\ref{lem:edgesEmanatingFromDensePairsII} for $(G, \DenseSpots, H',
 \mathcal A)$ yields a dense spot $D=(U,W;F)\in\DenseSpots$ and a set
 $A\in\mathcal A$, together with two sets $X\subset U\cap A\cap (A_H\setminus 
 V(\mathcal M))$, $Y\subset W\cap (B_H\setminus  V(\mathcal M))$ such that
 $|X|=|Y|>\alpha_\PARAMETERPASSING{L}{lem:edgesEmanatingFromDensePairsII}\nu
 k=\alpha\nu k$, and such that $(X,Y)$ is $\epsilon_\PARAMETERPASSING{L}{lem:edgesEmanatingFromDensePairsII}$-regular and has density at least $$\frac{\tau_\PARAMETERPASSING{L}{lem:edgesEmanatingFromDensePairsII}\rho_\PARAMETERPASSING{L}{lem:edgesEmanatingFromDensePairsII}}{4\Omega_\PARAMETERPASSING{L}{lem:edgesEmanatingFromDensePairsII}}=\frac{\tau\rho}{8\Omega}.$$ This contradicts the maximality of $\mathcal M$, proving~\eqref{bedingung2}.

In order to see~$\mathbf{(P2)}$, it suffices to observe that by~\eqref{bedingung2} and by Property~\eqref{duke} of Definition~\ref{tupelclass}, the set $V(\mathcal M)$ is incident with at least $\tau  kn-\frac\tau 2 kn=\frac\tau 2 kn$ edges. By Definition~\ref{tupelclass}~\eqref{maxroach}, it follows that $|V(\mathcal M)|\geq \frac\tau 2 kn\cdot \frac 1{\Omega k}\geq \frac{\tau}{2\Omega}n$, as desired.
\end{proof}

\subsection{Augmenting paths for matchings}

We now prove the main lemma of Section~\ref{sec:augmenting}, namely Lemma~\ref{lem:Separate}. We will use an augmenting path technique for our \semiregular matchings, similar to the augmenting paths commonly used for traditional matching theorems. For this, we need the following definitions.

\begin{definition}[\bf Alternating path, augmenting
path\index{general}{alternating path}\index{general}{augmenting path}]\label{altPath} 
Suppose that $n,s\in\NN$ and $\delta>0$. 
Given an $n$-vertex graph $G$,
and a \semiregular matching $\M$, we  call a sequence $\mathcal S=(Y_0,\A_1,Y_1, \A_2, Y_2,\ldots ,\A_h ,Y_h)$ (where $h\ge 0$ is arbitrary) a {\em $(\delta,s)$-alternating path for $\M$ from $Y_0$} if for all $i\in[h]$ we have 
 \begin{enumerate}[(i)]
\item $\A_i\subseteq \mathcal V_1(\mathcal M)$ and the sets $\mathcal A_i$ are pairwise disjoint, \label{alt1e}
\item  $Y_0\subseteq V(G) \setminus V(\mathcal M)$ and 
$Y_i=\bigcup_{(A,B)\in\M, A\in\A_i}B$,\label{alte2}
\item  $|Y_{i-1}|\geq\delta n$, and\label{alte3}
\item $e(A, Y_{i-1})\geq s\cdot |A|$,  for each $A\in\mathcal A_i$.\label{alte4}
\suspend{enumerate}
If in addition there is a set $\mathcal C$ of disjoint subsets of
$V(G)\setminus (Y_0\cup V(\M))$ such that 
\resume{enumerate}[{[(i)]}]
\item  $e(\bigcup \mathcal C,
Y_h)\geq t\cdot n$,\label{augm5}
\end{enumerate}
then we say that $\mathcal S'=(Y_0,\A_1,Y_1, \A_2, Y_2,\ldots ,\A_h ,Y_h,\mathcal C)$ is a \emph{$(\delta,s,t)$-augmenting path for $\M$ from~$Y_0$ to $\mathcal C$}.

The number $h$ is called the \index{general}{length of alternating path} \emph{length} of $\mathcal{S}$ (or of $\mathcal{S'}$).
\end{definition}

Next, we show that a \semiregular matching either has an augmenting path or admits a partition into two parts so that only few edges cross these parts in a certain way.

\begin{lemma}\label{lem:augmentingPath}
 Given an $n$-vertex graph $G$ with $\maxdeg(G)\le \Omega
k$, a number $\tau\in (0,1)$, a \semiregular matching $\M$, a set $Y_0\subseteq V(G)\setminus V(\M)$, and  a set $\C$ of disjoint subsets of $V(G)\setminus (V(\M)\cup Y_0)$, one of the following holds:
  \begin{itemize}
  \item[{\bf(M1)}] 
    There is a \semiregular matching $\mathcal
  M''\subseteq \mathcal M$ with $$e\left(\bigcup \mathcal C\cup
  V_1(\mathcal M\setminus\mathcal M''),Y_0\cup  V_2(\mathcal M'')\right)<\tau nk,$$ 
    \item[{\bf(M2)}] 
$\M$ has a $(\frac\tau{2\Omega}, \frac{\tau^2}{8\Omega}k,\frac{\tau^2}{16\Omega}
k)$-augmenting path of length at most $2\Omega/\tau$ from $Y_0$ to $\mathcal C$.
    \end{itemize}
\end{lemma}
\begin{proof}
If $|Y_0|\le \frac{\tau}{2\Omega} n$ then {\bf(M1)} is satisfied for $\M'':=\emptyset$. Let us therefore assume the contrary.
 
Choose a $(\frac\tau{2\Omega},\frac{\tau^2}{8\Omega} k)$-alternating path  $\mathcal
S=(Y_0,\A_1,Y_1, \A_2,$ $Y_2,\ldots,\A_h, Y_h)$ for $\mathcal M$
 with $|\bigcup_{\ell=1}^h\mathcal A_\ell|$ maximal.

 Now, let $\ell^*\in \{0,1,\ldots,h\}$ be maximal with $|Y_{\ell^*}|\geq \frac\tau{2\Omega} n$. Then $\ell^*\in\{h,h-1\}$. Moreover,
 as $|Y_\ell|\geq\frac\tau{2\Omega} n$ for all $\ell\leq\ell^*$, we have that $(\ell^*+1)\cdot \frac\tau{2\Omega} n\leq |\bigcup_{\ell\leq\ell^*} Y_\ell| \leq n$ and thus
\begin{equation}\label{ellstarbounded}
 \ell^*+1\leq \frac{2\Omega}\tau.
\end{equation}

Let $\mathcal M''\subseteq \mathcal M$ consist of all $\M$-edges $(A,B)\in \M$ with  $A\in\bigcup_{\ell\in[h]}\A_\ell$. Then, by the choice of~$\mathcal S$,
\begin{align} \nonumber
 e\left(V_1(\M\setminus \M''), \bigcup_{\ell=0}^{\ell^*}Y_{\ell}\right)&=\sum_{\ell=0}^{\ell^*}e\left(V_1(\M\setminus \M''), Y_{\ell}\right)\\ \label{endofthepath}
&  < (\ell^*+1)\cdot \frac{\tau^2}{8\Omega} k \cdot |V_1(\M\setminus \M'')|  \leByRef{ellstarbounded} \frac\tau 4 kn.
\end{align}
Furthermore, if $\ell^*=h -1$ (that is, if $|Y_{h}|<\frac\tau{2\Omega} n$) then  
\begin{equation}\label{deltanisnothing}
e\left(V_1(\M\setminus \M'')\cup\bigcup \mathcal C,Y_h\right)\ <\  \frac\tau{2\Omega} n\cdot\maxdeg(G)\ \leq \ \frac\tau{2\Omega}\Omega kn \ =\ \frac\tau 2kn.
\end{equation}

So, regardless of whether $h=\ell^*$ or $h=\ell^*+1$, we get from~\eqref{endofthepath} and~\eqref{deltanisnothing} that
\[
e\left(V_1(\M\setminus \M'')\cup\bigcup \mathcal C, Y_0\cup V_2(\M'')\right)<
\frac 34\tau kn + e\left(\bigcup \mathcal C, \bigcup_{\ell=0}^{\ell^*}Y_{\ell}\right).
\]

Thus, if $e(\bigcup \mathcal C,\bigcup_{\ell=0}^{\ell^*}Y_\ell)\leq \frac\tau 4kn$,  we see that~$\mathbf{(M1)}$ is satisfied for $\M''$. So, assume the contrary. 
Then, by~\eqref{ellstarbounded},
there is an index $j\in\{0,1,\ldots,\ell^*\}$ for which
\begin{equation*}\label{CYedges}
 e\left(\bigcup \mathcal C,Y_j\right)
 \ > \ \frac{\tau^2}{16\Omega}kn,
\end{equation*}
and thus, $(Y_0,\A_1,Y_1, \A_2,$ $Y_2,\ldots,\A_j, Y_j,\C)$ is a $(\frac\tau{2\Omega}, \frac{\tau^2}{8\Omega}
k, \frac{\tau^2}{16\Omega} k)$-augmenting path for $\mathcal M$.
This proves~{\bf(M2)}.
\end{proof}

The aim of this section is to find a \semiregular matching covering as many vertices from the graph as possible. This is done by iteratively improving a matching. Below, Lemma~\ref{lem:AugmentORSeparate} provides with such an iterative step: given a \semiregular matching $\M$ we either find~\eqref{it:AugmAss2} a better \semiregular matching $\M'$, or there is~\eqref{it:AugmAss1} a natural barrier to finding such a matching. This barrier is a separation of the previous \semiregular matching into two blocks ($\M''$ and $\M\setminus \M''$) such that very few edges ``cross'' this separation. The absence of such a separation guarantees the existence of an augmenting path for $\M$, which can be used to find  a better \semiregular matching. This matching $\M'$  has~{\bf(C1)} to improve $\M$ substantially and~{\bf(C2)} respect the structure of the graph and of $\M$.

\begin{lemma}\label{lem:AugmentORSeparate}
For every
$\Omega\in \NN$ and  $\tau\in (0,\frac{1}{2\Omega})$ there is a number $\tau'\in (0,\tau)$ such that for every $\rho\in(0,1)$ there is a number
$\alpha\in (0,\tau'/2)$ such that for every $\epsilon\in (0,\alpha)$ there is a number
$\pi>0$ 
 such that for every $\gamma>0$
there is $k_0\in\NN$ such
that the following holds for every $k>k_0$ and every $h\in (\gamma k,k/2)$.

Let $G$ be a graph  of order $n$ with
$\maxdeg(G)\le \Omega k$, with an $(\eps^3,\rho,h)$-\semiregular matching~$\M$
 and with a $(\rho k, \rho)$-dense cover $\DenseSpots$ that  absorbs $\M$. Let $Y\subset
V(G)\setminus  V(\mathcal M)$, and let
$\mathcal C$ be an $h$-ensemble in $G$ with  $\mathcal C\cap (V(\mathcal M)\cup Y)=\emptyset$.
Assume that $U\cap C\in\{\emptyset,C\}$ for each $D=(U,W;F)\in\DenseSpots$ and each $C\in\mathcal C\cup \mathcal V_1(\mathcal M)$. 

Then one of the following holds.
\begin{enumerate}[(I)]
  \item\label{it:AugmAss1} There is a \semiregular matching $\mathcal
  M''\subseteq \mathcal M$ such that $$e\left(\bigcup \mathcal C\cup
  V_1(\mathcal M\setminus\mathcal M''),Y\cup  V_2(\mathcal M'')\right)<\tau nk.$$
\item\label{it:AugmAss2} There is an $(\epsilon,\alpha ,\pi h)$-\semiregular
matching $\mathcal M'$ such that
  \begin{itemize}
    \item[{\bf(C1)}] $|V(\mathcal M)\setminus V(\mathcal M')|\le
    \epsilon n$, and $
    |V(\mathcal M')|\geq  |V(\mathcal M)| + \frac{\tau'}2 n$, and
    \item[{\bf(C2)}] for each $(T,Q)\in\mathcal M'$ there are sets $C_1\in\mathcal V_1(\mathcal
 M)\cup \mathcal C$, $C_2\in \mathcal V_2(\mathcal  M)\cup\{Y\}$ and a dense
 spot $D=(U,W;F)\in\DenseSpots$ such that $T\subset C_1\cap U$ and $Q\subset
C_2\cap W$.
  \end{itemize}
\end{enumerate}
\end{lemma}
\begin{proof}
We divide the proof into five
steps.

\paragraph{Step 1: Setting up the parameters.}
Suppose that $\Omega$ and $\tau$ are given. 
For
$\ell=0,1,\ldots ,\lceil2\Omega/\tau\rceil$, we define the auxiliary para\-meters

\begin{equation}\label{auroraenpekin}
\tau^{(\ell)}:= \left(\frac{\tau^2}{32\Omega}\right)^{\lceil \frac{2\Omega}\tau\rceil -\ell +2}\;,
\end{equation}
and set $$\tau':=\frac{\tau^{(0)}}{2\Omega}.$$

Given $\rho$, we define 
$$ \alpha:=\frac{\tau'\rho}{16\Omega}.$$
Then, given $\epsilon$, 
for $\ell=0,1,\ldots ,\lceil2\Omega/\tau\rceil$, we define the further auxiliary parameters
$$\mu^{(\ell)}:= \alpha_\PARAMETERPASSING{L}{lem:edgesEmanatingFromDensePairsIII}\big(\Omega,\rho,\eps^3, \tau^{(\ell )}\big)$$
 which are given by
Lemma~\ref{lem:edgesEmanatingFromDensePairsIII} for
input parameters $\Omega_\PARAMETERPASSING{L}{lem:edgesEmanatingFromDensePairsIII}:=\Omega$, $\rho_\PARAMETERPASSING{L}{lem:edgesEmanatingFromDensePairsIII}:=\rho$, 
$\epsilon_\PARAMETERPASSING{L}{lem:edgesEmanatingFromDensePairsIII}:=\epsilon ^3$, and $\tau_\PARAMETERPASSING{L}{lem:edgesEmanatingFromDensePairsIII}:=\tau^{(\ell )}$.
Set
$$\pi:=\frac{\epsilon}{2}  \cdot \min\left\{\mu^{(\ell)}\::\:\ell=0,\ldots,\lceil 2\Omega/\tau\rceil\right\}\;.$$

Given the next\footnote{in the order of quantification from the statement of the lemma} input parameter $\gamma$, 
Lemma~\ref{lem:edgesEmanatingFromDensePairsIII} for parameters as above and the final input $\nu_\PARAMETERPASSING{L}{lem:edgesEmanatingFromDensePairsIII}:=\gamma$ yields $k_{0_\PARAMETERPASSING{L}{lem:edgesEmanatingFromDensePairsIII}}=:k_0^{(\ell)}$, set
$$k_0:=\max\left\{k_0^{(\ell)}\::\:\ell=0,\ldots,\lceil 2\Omega/\tau\rceil \right\}.$$

\paragraph{Step 2: Finding an augmenting path.}
We apply Lemma~\ref{lem:augmentingPath} to $G$, $\tau$, $\mathcal M$, $Y$ and $\mathcal C$. Since~$\mathbf{(M1)}$ corresponds to~\eqref{it:AugmAss1}, let us assume that the outcome of the lemma is $\mathbf{(M2)}$. Then there is a $(\frac\tau{2\Omega}, \frac{\tau^2}{8\Omega}k,\frac{\tau^2}{16\Omega} k)$-augmenting path $\mathcal S'=(Y_0,\A_1,Y_1, \A_2,$ $Y_2,\ldots,$
$\A_{j^*}, Y_{j^*}, \C)$ for $\mathcal M$ starting from $Y_0:=Y$
such that $j^* \leq 2\Omega/\tau$.
%
%

Our aim is now to show that~\eqref{it:AugmAss2} holds. 

\paragraph{Step 3: Creating parallel
matchings.}

Inductively, for $\ell=j^*,j^*-1,\dots ,0$ we shall define
auxiliary bipartite induced subgraphs
$H^{(\ell)}\subset G$  with colour classes $P^{(\ell)}$ and $Y_{\ell}$ that satisfy
\begin{enumerate}[(a)]
\item \label{Hell}
$e(H^{(\ell)})\ge\tau^{(\ell)} kn,$
\end{enumerate}
and $(\eps^3,2\alpha,\mu^{(\ell)}h)$-\semiregular matchings $\mathcal M^{(\ell)}$ that satisfy
 \begin{enumerate}[(a)]\setcounter{enumi}{1}
 \item $V_1(\mathcal M^{(\ell)})\subseteq P^{(\ell)}$,\label{liegtinP}
  \item for each $(A',B')\in\mathcal M^{(\ell)}$ there are a dense spot $(U,W;F)\in
  \DenseSpots$ and a set $A\in \V_1(\mathcal M)$ (or a set $A\in\mathcal C$ if $\ell=j^*$)  such that  $A'\subset U\cap A$  and $B'\subset W\cap Y_{\ell}$,\label{ribot}
  \item $|V(\mathcal M^{(\ell)})|\ge \frac{\tau^{(\ell)} }{2\Omega}n$, and\label{dasaltealpha}
  \item\label{dasneuef} $|B\cap V_2(\mathcal M^{(\ell)})|=|A\cap P^{(\ell -1)}|$ for each edge $(A,B)\in \mathcal M$, if $\ell>0$.
\end{enumerate}

We take $H^{(j^*)}$ as the induced
bipartite subgraph of $G$ with colour classes  $P^{(j^*)}:=\bigcup \mathcal C$
and $Y_{j^*}$. Definition~\ref{altPath}~\eqref{augm5} together with~\eqref{auroraenpekin} ensures~\eqref{Hell} for $\ell=j^*$.
 Now, for $\ell \leq j^*$, suppose $H^{(\ell)}$ is already defined. Further, if $\ell< j^*$ suppose also that $\mathcal M^{(\ell +1)}$ is already defined. We shall define $\mathcal M^{(\ell)}$, and, if $\ell >0$, we shall also define $H^{(\ell-1)}$.

Observe that  $(G,\DenseSpots,H^{(\ell)},\mathcal A_\ell)\in\mathcal G(n,k,\Omega,\rho,\frac{h}k,\tau^{(\ell)})$, because of~\eqref{Hell} and the assumptions of the lemma. So, 
applying Lemma~\ref{lem:edgesEmanatingFromDensePairsIII} to
  $(G,\DenseSpots,H^{(\ell)},\mathcal A_\ell)$  and noting that $\frac{\tau^{(\ell)}\rho}{8\Omega}\geq 2\alpha$, we obtain
  an
$(\eps^3,2\alpha,\mu^{(\ell)}h)$-\semiregular matching $\mathcal
M^{(\ell)}$ that satisfies conditions~\eqref{liegtinP}--\eqref{dasaltealpha}.

If $\ell>0$, we define $H^{(\ell-1)}$ as follows. 
For each $(A,B)\in \mathcal M$ take a set $\tilde A\subset A$ of
cardinality $|\tilde A|=|B\cap V(\mathcal M^{(\ell)})|$ so that 
\begin{equation}\label{schoengross}
e(\tilde A,Y_{\ell-1})\ \ge\ \frac{\tau^2}{8\Omega}k\cdot |\tilde A| \;.
\end{equation}
This is possible by Definition~\ref{altPath}~\eqref{alte4}: just choose those vertices from $A$ for $\tilde A$ that send most edges to $Y_{\ell -1}$.
Let $P^{(\ell-1)}$ be the union of all the sets $\tilde A$.
Then, \eqref{dasneuef} is satisfied. Furthermore,
$$|P^{(\ell-1)}|\ =\ |V_2(\mathcal M^{(\ell)})|
\ \overset{\eqref{dasaltealpha}}\ge \ \frac{\tau^{(\ell)}}{4\Omega}n.$$ 
So, by~\eqref{schoengross},
\begin{equation}\label{KantOrGoethe}
e(P^{(\ell-1)},Y_{\ell-1})\ \ge\ \frac{\tau^2}{8\Omega} k \cdot |P^{(\ell-1)}|\  \geq \ \frac{\tau^2\cdot \tau^{(\ell)}}{32\Omega^2}kn \ \overset{\eqref{auroraenpekin}}= \ \tau^{(\ell-1)}kn\;.
\end{equation}

We let $H^{(\ell -1)}$ be the
bipartite subgraph of $G$ induced by the colour classes $P^{(\ell-1)}$ and $Y_{\ell-1}$. 
 Then~\eqref{KantOrGoethe}
establishes~\eqref{Hell} for $H^{(\ell -1)}$. This finishes step $\ell$.\footnote{Recall that  the matching $\mathcal M^{(\ell -1)}$ is only to be defined in step $\ell-1$.}

\paragraph{Step 4: Harmonising the matchings.}
Our \semiregular matchings $\mathcal
M^{(0)},\ldots,\mathcal M^{(j^*)}$ will be a good base for constructing the \semiregular matching $\mathcal M'$ we are after. However, we do not know anything about  $|B\cap V_2(\mathcal M^{(\ell)})| - |A\cap V_1(\mathcal M^{(\ell -1)})|$ for the $\M$-edges $(A,B)\in \mathcal M$.
But this term will be crucial in determining how much of $V(\M)$ gets lost when we replace some of its $\M$-edges with $\bigcup \M^{(\ell)}$-edges. For this reason, we refine $\mathcal M^{(\ell)}$ in a way that its $\mathcal M^{(\ell)}$-edges become almost equal-sized.

Formally, we shall inductively construct
\semiregular matchings $\mathcal N^{(0)},\ldots, \mathcal
N^{(j^*)}$ such that  for $\ell=0,\ldots ,j^*$ we have
\begin{enumerate}[(A)]
\item \label{eq:NellSemiregular} $\mathcal N^{(\ell)}$ is an
$(\epsilon,\alpha,\pi h)$-\semiregular matching,
\item\label{absorbi}$\mathcal M^{(\ell)}$ absorbs $\mathcal N^{(\ell)}$,
\item \label{dasneueF} if $\ell>0$ and $(A,B)\in\mathcal M$ with $A\in\mathcal A_\ell$ then $|A\cap  V(\mathcal N^{(\ell-1)})|\ge |B \cap V(\mathcal N^{(\ell)})| $, and
\item \label{eq:SizeNInduc} $|V_2(\mathcal N^{(\ell)})|\ge |V_1(\mathcal
N^{(\ell-1)})| - \frac{\epsilon}{2}\cdot |V_2(\mathcal M^{(\ell)})|$  if
$\ell>0$ \ and $|V_2(\mathcal N^{(0)})|\ge \frac{\tau^{(0)}}{2\Omega} n = \tau'
n$.
\end{enumerate}

Set $\mathcal N^{(0)}:=\mathcal M^{(0)}$. Clearly~\eqref{absorbi} holds for
$\ell=0$,~\eqref{eq:NellSemiregular} is easy to check, and~\eqref{dasneueF} is void. Finally, Property~\eqref{eq:SizeNInduc} holds because of~\eqref{dasaltealpha}.
Suppose now $\ell>0$ and that we already constructed matchings
$\mathcal N^{(0)},\ldots,\mathcal N^{(\ell-1)}$ satisfying Properties~\eqref{eq:NellSemiregular}--\eqref{eq:SizeNInduc}.

 Observe that for any $(A,B)\in\mathcal M$ we have that
\begin{equation}\label{thisisenough}
|B\cap  V_2(\mathcal M^{(\ell)})| \
\overset{\eqref{liegtinP}, \eqref{dasneuef}}\geq \ |A\cap V_1(\mathcal M^{(\ell-1)})| \
\ge\ |A\cap V_1(\mathcal N^{(\ell-1)})|,
\end{equation} 
where the last inequality holds because of~\eqref{absorbi} for $\ell -1$.

So, we can choose a subset $X^{(\ell)}\subseteq V_2(\mathcal M^{(\ell)})$ such that 
$|B\cap X^{(\ell)}|=|A\cap V(\mathcal N^{(\ell-1)})|$ for each $(A,B)\in\mathcal M$.
Now, for each $(S,T)\in \mathcal M^{(\ell)}$ write $\widehat{T}:=T\cap X^{(\ell)}$, and choose a subset $\widehat{S}$ of $S$ of size $|\widehat{T}|$. 
 Set
\begin{equation}\label{eq:NelDef}
\mathcal{N}^{(\ell)}:=\left\{(\widehat{S},\widehat{T})\::\:(S,T)\in \mathcal
M^{(\ell)}, |\widehat{T}|\geq \frac{\epsilon}{2}\cdot |T| \right\}. 
\end{equation}
Then~\eqref{absorbi}  and~\eqref{dasneueF} hold for $\ell$.

For~\eqref{eq:NellSemiregular}, note that Fact~\ref{fact:BigSubpairsInRegularPairs} implies that $\mathcal N^{(\ell)}$ is an $\left(\eps,2\alpha-\eps^3,\frac{\epsilon}{2}\mu^{(\ell)}h\right)$-\semiregular matching.

In order to verify~\eqref{eq:SizeNInduc}, it suffices to observe that
\begin{align*}
|V_2(\mathcal N^{(\ell)})|
&=\sum_{(\widehat S,\widehat T)\in\mathcal N^{(\ell)}} |\widehat{T}|\geq   | X^{(\ell)}|\  -\ \sum_{(S,T)\in\mathcal M^{(\ell)} } \frac{\epsilon}{2} \cdot |T|\\[8pt]
&\geq  \sum_{(A,B)\in\mathcal M} |A\cap V_1(\mathcal N^{(\ell-1)})|\ -\ \frac{\epsilon}{2} \cdot |V_2(\mathcal M^{(\ell)} )|
= |V_1(\mathcal N^{(\ell-1)})|\ -\ \frac{\epsilon }{2} \cdot |V_2(\mathcal M^{(\ell)} )|.
\end{align*}

\paragraph{Step 5: The final matching.}
Suppose that $(A,B)\in\mathcal M$ with $A\in\mathcal A_\ell$ for some $\ell\in\{1,2,\ldots,j^*\}$. Then, set $A':=A\setminus
V_1(\mathcal N^{(\ell-1)})$. Also, choose a set $B'\subset B\setminus V_2(\mathcal N^{(\ell)})$ of cardinality $|A'|$. This is possible by~\eqref{dasneueF}. By~\eqref{eq:NelDef} we deduce that 
\begin{equation}\label{eq:B'Size}
|B\setminus V_2(\mathcal N^{(\ell)})|-|B'|\le \frac{\varepsilon}{2}|B|\;.
\end{equation}
  We consider the set $\mathcal L\subset \mathcal M$ consisting of all $\M$-edges $(A,B)\in\mathcal M$ with $|A'|>\frac{\epsilon}{2} \cdot |A|$. 

Set 
$$\mathcal K:=\{(A', B'):(A,B)\in\mathcal L\}.$$ 
By the assumption of the lemma, for every $(A',B')\in\mathcal K$
there are an edge $(A,B)\in\mathcal M$ and a dense spot $D=(U,W;F)\in\DenseSpots$ such that 
\begin{equation}\label{lacolegiala}
\text{$A'\subset A\subset U$ and
$B'\subset B\subset W$.}
\end{equation}
 Since $\mathcal M$ is $(\eps^3,\rho,h)$-\semiregulars, Fact~\ref{fact:BigSubpairsInRegularPairs} implies that $\mathcal K$ is an $(\eps,\rho-\eps^3,\frac{\epsilon}{2} h)$-\semiregular matching. Set $$\mathcal M':=\mathcal
K\cup\bigcup_{\ell=0}^{j^*}\mathcal N^{(\ell)}.$$ It is easy to check that $\mathcal M'$ is an
$(\epsilon,\alpha,\pi h)$-\semiregular matching. Using~\eqref{lacolegiala} together with~\eqref{absorbi} and \eqref{ribot}, we see that  {\bf(C2)} holds for $\mathcal M'$. 

In order to see {\bf(C1)}, we calculate
\begin{align*}
&|V(\mathcal M) \setminus V(\mathcal M')|=\sum_{(A,B)\in\mathcal M}\Big(|A\setminus V_1(\cup_{\ell=0}^{j^*}\mathcal N^{(\ell)}\cup \mathcal K)|+|B\setminus V_2(\cup_{\ell=0}^{j^*}\mathcal N^{(\ell)}\cup \mathcal K)|\Big)\\
&\leByRef{eq:B'Size} \underbrace{\sum_{(A,B)\in\mathcal M\setminus \mathcal L}\left(|A'\cup B'|+\frac{\varepsilon}{2}|B|\right)}_{\texttt{(sum1)}}\;+\; 
\underbrace{\sum_{(A,B)\in\mathcal L}\ \Big( |A\setminus V_1(\cup_{\ell=1}^{j^*}\mathcal N^{(\ell-1)}\cup \mathcal K))| + |B\setminus V_2(\cup_{\ell=1}^{j^*}\mathcal N^{(\ell)}\cup \mathcal K)| \Big)}_{\texttt{(sum2)}}\;.
\end{align*}
In \texttt{(sum2)}, consider an arbitrary term corresponding to $(A,B)$. By the definition of $\mathcal K$, the term $|A\setminus V_1(\cup_{\ell=1}^{j^*}\mathcal N^{(\ell-1)}\cup \mathcal K)|$ is zero. To treat the term $|B\setminus V_2(\cup_{\ell=1}^{j^*}\mathcal N^{(\ell)}\cup \mathcal K)|$, we recall that $|A|=|B|$ and $|A'|=|B'|$ (in the definition of $\mathcal K$). This gives that $|B\setminus V_2(\cup_{\ell=1}^{j^*}\mathcal N^{(\ell)}\cup \mathcal K)|=|A \cap V_1(\cup_{\ell=1}^{j^*}\mathcal N^{(\ell-1)}) | - |B \cap V_2(\cup_{\ell=1}^{j^*}\mathcal N^{(\ell)})|$. This leads to
\begin{align*}
|V(\mathcal M) \setminus V(\mathcal M')|&\le
\sum_{(A,B)\in\mathcal M\setminus \mathcal L}\left(|A'\cup B'|+\frac{\varepsilon}{2}|B|\right)\\
&~~~~+\ 
\sum_{(A,B)\in\mathcal L}\ \sum_{\ell=1}^{j^*} \Big( |A\cap V_1(\mathcal N^{(\ell -1)})| - |B\cap V_2(\mathcal N^{(\ell)})| \Big)\\[10pt]
&\le 
  \sum_{(A,B)\in\mathcal M\setminus \mathcal L}\left(\frac{\epsilon}{2}|A|+\varepsilon|B|\right) \ 
+ \ 
\sum_{\ell=1}^{j^*} \Big( |V_1(\mathcal N^{(\ell -1)})| - |V_2(\mathcal N^{(\ell)})| \Big) \\[10pt]
&\overset{\eqref{eq:SizeNInduc}}\le
\frac{3\epsilon}{4}   n\ 
+\ 
\sum_{\ell=1}^{j^*}\frac{\epsilon }{2} \cdot |V_2(\mathcal M^{(\ell)} )|\le
\epsilon  n\;.
\end{align*}

Using the fact that $V_2(\mathcal N^{(0)})\subseteq V(\mathcal M')\setminus V(\mathcal M)$ the last calculation also implies that
 \begin{align*} 
|V(\mathcal M')|-|V(\mathcal M)|
&\ge |V_2(\mathcal N^{(0)})|-|V(\mathcal M) \setminus V(\mathcal M')|\overset{\eqref{eq:SizeNInduc}}\ge
\tau' n -\epsilon n>\frac{\tau'}2 n\;,
\end{align*}
since $\eps<\alpha\leq \tau'/2$ by assumption.
\end{proof}

\medskip

Iterating Lemma~\ref{lem:AugmentORSeparate} we prove the main result of the section.

\begin{lemma}\label{lem:Separate}
For every
$\Omega\in \NN$ and $\rho\in (0,1/\Omega)$ there exists a number $\beta>0$ such that for
every $\epsilon\in (0,\beta)$, 
there are
$\epsilon',\pi>0$ such that for each $\gamma>0$ there exists a number $k_0\in\NN$ such that  the following holds for
every $k>k_0$ and $c\in(\gamma k,k/2)$.

Let $G$ be a graph  of order $n$, with $\maxdeg(G)\le \Omega k$. Let 
$\DenseSpots$ be a $(\rho k, \rho)$-dense cover of $G$, and let $\mathcal M$ be
an $(\epsilon',\rho,c)$-\semiregular matching that is absorbed by $\DenseSpots$. Let
$\mathcal C$ be a $c$-ensemble in $G$ with $\mathcal C\cap (V(\mathcal M))=\emptyset$. 
Let $Y\subset V(G)\setminus (V(\mathcal M)\cup
 \bigcup\mathcal C)$.
Assume
that for each $(U,W;F)\in\DenseSpots$, and for each $C\in\mathcal V_1(\mathcal M)\cup \mathcal C$ we have that 
 \begin{equation}\label{eq:510ass1}
 U\cap C\in\{\emptyset,C\}\;.
 \end{equation} 

Then there exists an
$(\epsilon,\beta,\pi c)$-\semiregular matching $\mathcal
M'$ such that
  \begin{enumerate}[(i)]
    \item\label{it:Sep1} $|V(\mathcal M)\setminus V(\mathcal M')|\le
    \epsilon n$,
    \item\label{it:Sep2} 
 for each $(T,Q)\in\mathcal M'$ there are sets $C_1\in\mathcal V_1(\mathcal
 M)\cup \mathcal C$, $C_2\in \mathcal V_2(\mathcal  M)\cup\{Y\}$ and a dense
 spot $D=(U,W;F)\in\DenseSpots$ such that $T\subset C_1\cap U$ and $Q\subset
C_2\cap W$, and
    \item\label{it:Sep3} $\mathcal M'$ can be partitioned into $\mathcal M_1$
    and $\mathcal M_2$ so that
    $$e\left( (\bigcup \mathcal C\cup V_1(\mathcal M )) \setminus   V_1(\mathcal M_1) \  , \ (Y\cup V_2(\mathcal M)) \setminus V_2(\mathcal M_2) \right)\  < \ \rho kn\;.$$
  \end{enumerate}
\end{lemma}

\begin{proof}
Let $\Omega$ and $\rho$ be given. Let $\tau':=\tau'_\PARAMETERPASSING{L}{lem:AugmentORSeparate} $ be the output given by Lemma~\ref{lem:AugmentORSeparate} for input
parameters $\Omega_\PARAMETERPASSING{L}{lem:AugmentORSeparate} :=\Omega$ and  $\tau_\PARAMETERPASSING{L}{lem:AugmentORSeparate} :=\rho/2$.

Set $\rho^{(0)}:=\rho$, 
set $L:=\lceil 2/\tau'\rceil +1$, and for $\ell\in [L]$, 
 inductively define
$\rho^{(\ell)}$ to be the output $\alpha_\PARAMETERPASSING{L}{lem:AugmentORSeparate} $ given by Lemma~\ref{lem:AugmentORSeparate} for the further input
parameter $\rho_\PARAMETERPASSING{L}{lem:AugmentORSeparate} :=\rho^{(\ell-1)}$ (keeping $\Omega_\PARAMETERPASSING{L}{lem:AugmentORSeparate}=\Omega $ and  $\tau_\PARAMETERPASSING{L}{lem:AugmentORSeparate}=\rho/2$ fixed).
Then $\rho^{(\ell+1)}\leq\rho^{(\ell)}$ for all $\ell$.
 Set
$\beta:=\rho^{(L)}$.

Given $\epsilon<\beta$
 we set
$\epsilon^{(\ell)}:=(\epsilon/2)^{3^{L-\ell}}$ for $\ell\in[L]\cup\{0\}$, and 
set 
$\epsilon':=\epsilon^{(0)}$.
Clearly,
\begin{equation}\label{summederepsilons}
\sum_{\ell=0}^L\epsilon^{(\ell)} \ \leq \ 
 \epsilon.
\end{equation}

Now, for $\ell+1\in[L]$, 
let $\pi^{(\ell)}:=\pi_\PARAMETERPASSING{L}{lem:AugmentORSeparate}$ be given by Lemma~\ref{lem:AugmentORSeparate} for input parameters $\Omega_\PARAMETERPASSING{L}{lem:AugmentORSeparate} :=\Omega$, $\tau_\PARAMETERPASSING{L}{lem:AugmentORSeparate} :=\rho/2$, $\rho_\PARAMETERPASSING{L}{lem:AugmentORSeparate} :=\rho^{(\ell)}$ and $\eps_\PARAMETERPASSING{L}{lem:AugmentORSeparate} :=\epsilon^{(\ell+1)}$.  For $\ell\in[L]\cup\{0\}$, set $\Pi^{(\ell)}:=\frac{\rho}{2\Omega}\prod_{j=0}^{\ell-1}\pi^{(j)}$. Let $\pi:=\Pi^{(L)}$.

Given $\gamma$, let $k_0$ be the maximum of the lower bounds $k_{0_\PARAMETERPASSING{L}{lem:AugmentORSeparate}}$ given by Lemma~\ref{lem:AugmentORSeparate} for input parameters $\Omega_\PARAMETERPASSING{L}{lem:AugmentORSeparate} :=\Omega$, $\tau_\PARAMETERPASSING{L}{lem:AugmentORSeparate} :=\rho/2$, $\rho_\PARAMETERPASSING{L}{lem:AugmentORSeparate} :=\rho^{(\ell -1)}$, $\eps_\PARAMETERPASSING{L}{lem:AugmentORSeparate} :=\epsilon^{(\ell)}$, $\gamma_\PARAMETERPASSING{L}{lem:AugmentORSeparate}:=\gamma\Pi^{(\ell)}$,  for $\ell\in[L]$.

\medskip

Suppose now we are given $G$, $\mathcal D$, $\mathcal C$, $Y$ and $\mathcal M$. Suppose further that $c>\gamma k>\gamma k_0$. 
Let $\ell\in\{0,1,\ldots,L\}$ be maximal subject to the condition that there is a matching $\mathcal M^{(\ell)}$ with the following properties:

\begin{enumerate}[(a)]
 \item\label{MellowYellow}
 $\mathcal M^{(\ell)}$ is an $( \epsilon^{(\ell)}, \rho^{(\ell)},\Pi^{(\ell)} c)$-\semiregular matching,
\item 
 \label{eq:NatImpAl}
$|V(\mathcal M^{(\ell)})|\ge \ell\cdot \frac{\tau'}2 n$,
\item
\label{epsipepsi}
$|V(\mathcal M)\setminus V(\mathcal M^{(\ell)})|\le \sum_{i=0}^\ell\epsilon^{(i)} n$, and
    \item\label{dontworryaboutMell} 
 for each $(T,Q)\in\mathcal M^{(\ell)}$ there are sets $C_1\in\mathcal V_1(\mathcal
 M)\cup \mathcal C$, $C_2\in \mathcal V_2(\mathcal  M)\cup\{Y\}$ and a dense
 spot $D=(U,W;F)\in\DenseSpots$ such that $T\subset C_1\cap U$ and $Q\subset
C_2\cap W$.
\end{enumerate}

Observe that such a number $\ell$ exists, as for $\ell=0$ we may take $\mathcal M^{(0)}=\mathcal M$. Also note that $\ell\leq 2/\tau' <L$ because of~\eqref{eq:NatImpAl}.

We now  apply Lemma~\ref{lem:AugmentORSeparate} with input parameters
$\Omega_\PARAMETERPASSING{L}{lem:AugmentORSeparate}:=\Omega$, $\ \tau_\PARAMETERPASSING{L}{lem:AugmentORSeparate}:=\rho/2$, $\ \rho_\PARAMETERPASSING{L}{lem:AugmentORSeparate}:=\rho^{(\ell)}$, $\ \epsilon_\PARAMETERPASSING{L}{lem:AugmentORSeparate}:=\epsilon^{(\ell+1)}<\beta\leq \rho^{(\ell +1)}=\alpha_\PARAMETERPASSING{L}{lem:AugmentORSeparate} $, 
$\ \gamma_\PARAMETERPASSING{L}{lem:AugmentORSeparate}:=\gamma\Pi^{(\ell)}$ to the  graph
$G$ with the $(\rho^{(\ell)} k,\rho^{(\ell)})$-dense cover $\DenseSpots$, the 
$( \epsilon^{(\ell)},  \rho^{(\ell)},\Pi^{(\ell)}c)$-\semiregular matching $\mathcal M^{(\ell)}$, the set 
$$\widetilde{Y}:=(Y\cup V_2(\mathcal M))\setminus V_2(\mathcal M^{(\ell)}),$$
and the $(\Pi^{(\ell)}c)$-ensemble
$$\widetilde{\mathcal C}:=  \left\{ C\setminus
V(\mathcal M^{(\ell)}) \ : \ C\in \V_1(\mathcal M)\cup \mathcal C,\  |C\setminus
V_1(\mathcal M^{(\ell)})|\ge
\Pi^{(\ell)}c\right\}.$$ 

Lemma~\ref{lem:AugmentORSeparate} yields a \semiregular matching  which either corresponds to $\mathcal M''_\PARAMETERPASSING{L}{lem:AugmentORSeparate}$ as in Assertion~\eqref{it:AugmAss1} or to $\mathcal M'_\PARAMETERPASSING{L}{lem:AugmentORSeparate}$ as in Assertion~\eqref{it:AugmAss2}.
Note that in the latter case, the matching $\mathcal M'_\PARAMETERPASSING{L}{lem:AugmentORSeparate}$ actually constitutes an $( \epsilon^{(\ell+1)},  \rho^{(\ell+1)},\Pi^{(\ell+1)}c)$-\semiregular matching $\mathcal M^{(\ell+1)}$ fulfilling all the above properties for $\ell+1\leq L$. In fact,~\eqref{eq:NatImpAl} and~\eqref{epsipepsi} hold for $\M^{(\ell+1)}$ because of $\mathbf{(C1)}$, and it is not difficult to deduce~\eqref{dontworryaboutMell} from $\mathbf{(C2)}$ and from~\eqref{dontworryaboutMell} for $\ell$. But this contradicts the choice of $\ell$. We conclude that we obtained a \semiregular matching $\mathcal M''_\PARAMETERPASSING{L}{lem:AugmentORSeparate}\subseteq \mathcal M^{(\ell)}$ as in Assertion~\eqref{it:AugmAss1} of Lemma~\ref{lem:AugmentORSeparate}.

Thus, in other words, $\mathcal M^{(\ell)}$ can be partitioned into $\mathcal M_1$ and $\mathcal M_2$
so that
\begin{equation}\label{eq:LifeIsNice}
e\Big(\bigcup \widetilde\C\cup V_1(\mathcal M_2)\ , \ \widetilde Y\cup
V_2(\mathcal M_1) \Big)\ <\ \tau_\PARAMETERPASSING{L}{lem:AugmentORSeparate} kn\ =\ \rho kn/2.
\end{equation}
 Set $\mathcal M':=\mathcal M^{(\ell)}$. Then $\mathcal M'$ is $(\epsilon,\beta,\pi c)$-\semiregular by~\eqref{MellowYellow}.
Note that Assertion~\eqref{it:Sep1} of the lemma holds by~\eqref{summederepsilons} and by~\eqref{epsipepsi}.  Assertion~\eqref{it:Sep2}  holds because of~\eqref{dontworryaboutMell}. 

Since
$$(Y\cup V_2(\mathcal M)) \setminus V_2(\mathcal M_2) \ \subseteq \ \widetilde Y\cup V_2(\mathcal M_1),$$
and because of~\eqref{eq:LifeIsNice} we know that in order to prove Assertion~\eqref{it:Sep3} it suffices to show that
\begin{align*}
X\ := \ & 
\big((\bigcup \mathcal C\cup V_1(\mathcal M )) \setminus   V_1(\mathcal M_1) \big) \setminus\big(\bigcup\widetilde\C\cup V_1(\mathcal M_2)\big)
\\
\ = \ &
\big(\bigcup \mathcal C\cup V_1(\mathcal M)\big) \setminus\big(\bigcup\widetilde\C\cup V_1(\mathcal M^{(\ell)})\big)
\end{align*}
 sends at most $ \rho kn/2$ edges to the rest of the graph. 
For this, it would be enough to see that $|X| \leq \frac\rho{2\Omega}n$, since by assumption, $G$ has maximum degree $\Omega k$.

To this end, note that
by assumption, $| \V_1(\mathcal M)\cup \mathcal C|\le \frac{n}{c}$. 
Further,
the definition of $\widetilde\C$ implies
that for each $A\in \mathcal C\cup \V_1(\mathcal M)$ we have 
that 
$|A\setminus \big(\bigcup\widetilde \C\cup V_1(\mathcal
M^{(\ell)}\big)|\le \Pi^{(\ell)}c$.
Combining these two observations, we obtain that
$$|X| 
< \Pi^{(\ell)}n
\le \frac \rho{2\Omega}n\;,$$ 
as desired.

\end{proof}

\section{Rough structure of LKS graphs}\label{sec:LKSStructure}
In this section we give a structural result for graphs
$G\in\LKSsmallgraphs{n}{k}{\eta}$, stated in Lemma~\ref{prop:LKSstruct}. Similar structural results were essential
also for proving Conjecture~\ref{conj:LKS} in the dense setting
in~\cite{AKS95,PS07+}. There, a certain matching structure was proved to exist
in the cluster graph of the host graph. This matching structure then allowed us to
embed a given tree into the host graph. We motivate the structure asserted by Lemma~\ref{prop:LKSstruct} in more detail in Section~\ref{ssec:motivation}.

Naturally, in our possibly sparse setting the sparse decomposition $\class$ of $G$ will enter the picture (instead of just the cluster graph of $G$. For more on sparse decomposition, see~\cite{cite:LKS-cut0}). 
There is an important subtlety though: we may need to ``re-regularize'' the cluster graph $\BGblack$ of $\class$.
In this case, we have to find another
regularization of parts of $G$, partially based on $\Gblack$. Lemma~\ref{lem:Separate} is the main tool to this end. The re-regularization is captured by the \semiregular matchings $\M_A$ and $\M_B$.

Let us note that this step is one
of the biggest differences between our approach and the announced solution of
the Erd\H os--S\'os Conjecture by Ajtai, Koml\'os, Simonovits and Szemer\'edi. In
other words, the nature of the graphs arising in the Erd\H os--S\'os Conjecture
allows a less careful approach with respect to regularization, still yielding a
structure suitable for embedding trees. We discuss the necessity of this step in further detail
in Section~\ref{ssec:whyaugment}. The main result of this paper Lemma~\ref{prop:LKSstruct}, is given in Section~\ref{sec6first}.

\subsection{Motivation for and intuition behind Lemma~\ref{prop:LKSstruct}}
\label{ssec:motivation}
Recall that \cite[Lemma~\ref{p0.lem:LKSsparseClass}]{cite:LKS-cut0} asserts that each graph $G=G_\PARAMETERPASSING{T}{thm:main}$ satisfying the conditions of Theorem~\ref{thm:main} has a sparse decomposition which captures almost all its edges. With this pre-processing at hand, we want Lemma~\ref{prop:LKSstruct} to provide specific structural properties of $G$  under which we could make the embedding of the tree $T=T_\PARAMETERPASSING{T}{thm:main}$ work. The complexity of these assertions (which span more than half a page) stems from the complicated nature of the sparse decomposition, and from the delicate features of the embedding techniques (worked out in \cite[Section~\ref{p3.sec:embed}]{cite:LKS-cut3}). 
In this section we try to explain and motivate the key assertions of Lemma~\ref{prop:LKSstruct}. The reader may skip the section at his or her convenience. The only bit from this section needed for the main result is Definition~\ref{def:XAXBXC}.

At this stage, let us introduce informally the notion of fine partition which we use to cut up the tree $T$. Let $\tau\ll 1$. We find a constant number of \emph{cut vertices} of $T$ so that the components (which we refer to as \emph{shrubs}) in the remainder of $T$ are of order at most $\tau k$. The cut vertices will decompose into two sets $W_A$ and $W_B$ so that the distance from any vertex of $W_A$ to any vertex in $W_B$ is odd. It can be shown that we can do the cutting so that each shrub either neighbours only one cut vertex from $W_A\cup W_B$, or it neighbours two, in which case both these cut vertices are in $W_A$. Thus, the set of all shrubs can be decomposed as $\shrubA\dcup\shrubB$ depending on the cut vertices that surround individual shrubs. The last property of the fine partition we shall use is that
\begin{equation}
\label{eq:shrubswitch}
\sum_{t\in \shrubA}v(t)\ge \sum_{t\in \shrubB}v(t)\;.
\end{equation}
The quadruple 
$(W_A,W_B,\shrubA,\shrubB)$ is then called a \emph{$(\tau k)$-fine partition} of $T$. The full definition which includes several additional properties is given in \cite[Section~\ref{p3.ssec:cut}]{cite:LKS-cut3}.

As said earlier, Lemma~\ref{prop:LKSstruct} is an extensive generalization of previous structural results on the LKS Conjecture in the dense setting. So, as a starting point for our motivation, let us explain the structural result Piguet and Stein~\cite{PS07+} use to prove the dense approximate version of the LKS Conjecture.
\begin{theorem}[\cite{PS07+}]\label{thm:PigSte}
For every $\eta>0$ and $q>0$ there exists a number $n_0$ such that for every $n>n_0$ and $k>qn$ we have the following. For every graph $G\in\LKSgraphs{n}{k}{\eta}$ contains each tree from $\treeclass{k}$.
\end{theorem}
Here, of course, the structure we work with is encoded in the cluster graph (in the sense of the original regularity lemma) $\BGblack$ of the graph $G_\PARAMETERPASSING{T}{thm:PigSte}$. Note that $\BGblack\in \LKSgraphs{N}{K}{\eta/2}$, where $N$ is the number of clusters and $K=k\cdot\frac Nn$. The main structural result of Piguet and Stein then reads as follows.
\begin{informallemma}[{\cite[Lemma~8]{PS07+}}, simplified]\label{lem:StructPigSt}
Suppose that $\BGblack\in \LKSgraphs{N}{K}{\alpha}$ and let us write $\BL=\largevertices{K}{\alpha}{\BGblack}$. Then we have at least one of the following two cases.
\begin{enumerate}
	\item[$\mathbf{(H1)}$] There exists a matching $M\subset \BGblack$ and an edge $A_1A_2$ so that  $\deg_{\BGblack}(A_i, \BL\cup V(M))\ge K$, for $i=1,2$.
	\item[$\mathbf{(H2)}$] There exists a matching $M\subset \BGblack$ and an edge $AB$ with $\deg_{\BGblack}(A, \BL\cup V(M))\ge K$, and  $\deg_{\BGblack}(B,\BL\cup V(M))\ge \frac K2$. Further, 
	\begin{equation}\label{eq:additional}
	\text{for every $e\in M$,}\quad
	|\neighbour_{\BGblack}(A)\cap e|\le 1\;.
	\end{equation}
\end{enumerate}
\end{informallemma}
Piguet and Stein use structures~\textbf{(H1)} and~\textbf{(H2)} to embed any
given tree $T\in\treeclass{k}$ into $G$ using the regularity method. A comprehensive description of the embedding procedure is given in
Sections~3.6 and~3.7 in~\cite{PS07+}. The embedding itself is quite technical but it follows a relatively pedestrian strategy which we present next. The regularity method tells us that a regular pair can be filled up by an arbitrary family of shrubs, provided that the colour classes of these shrubs (viewed as one bipartite graph) do not overfill the end-clusters of that regular pair. The degree conditions in Informal Lemma~\ref{lem:StructPigSt} suggest that we will utilize the clusters of $M$ and of $\BL$. More precisely, some of the shrubs will be accommodated in the edges of the matching $M$. Suppose next that we would like to proceed with embedding some shrubs using a cluster $X\in\BL$. This can be done as follows. Using the high-degree property of $X$ we can find a cluster $Y$ adjacent to $X$ that is not filled up completely by the image of $T$. We then use the pair $XY$ to accommodate further shrubs.
We keep embedding $T$ by mapping $W_A$ to $A_1$ (in $\mathbf{(H1)}$) or to $A$ (in $\mathbf{(H2)}$), $W_B$ to $A_2$ or to $B$, and the shrubs pendent from these cut-vertices either into the regular edges of $M$, or to edges incident to clusters $\BL$ as described above. Thus, the degree conditions in Informal Lemma~\ref{lem:StructPigSt} guarantee that we can accommodate shrubs of total order up to $k$ from $A_1$, $A_2$, and $A$. The degree bound for $B$ guarantees that we can embed shrubs of total order up to $k/2$ from $B$, recall that this is sufficient, thanks to~\eqref{eq:shrubswitch}. The fact that $A_1A_2$ or $AB$ forms an edge allows us to make connections between images of $W_A$ and $W_B$.

\begin{figure}[t]
	\centering
   \includegraphics[scale=0.8]{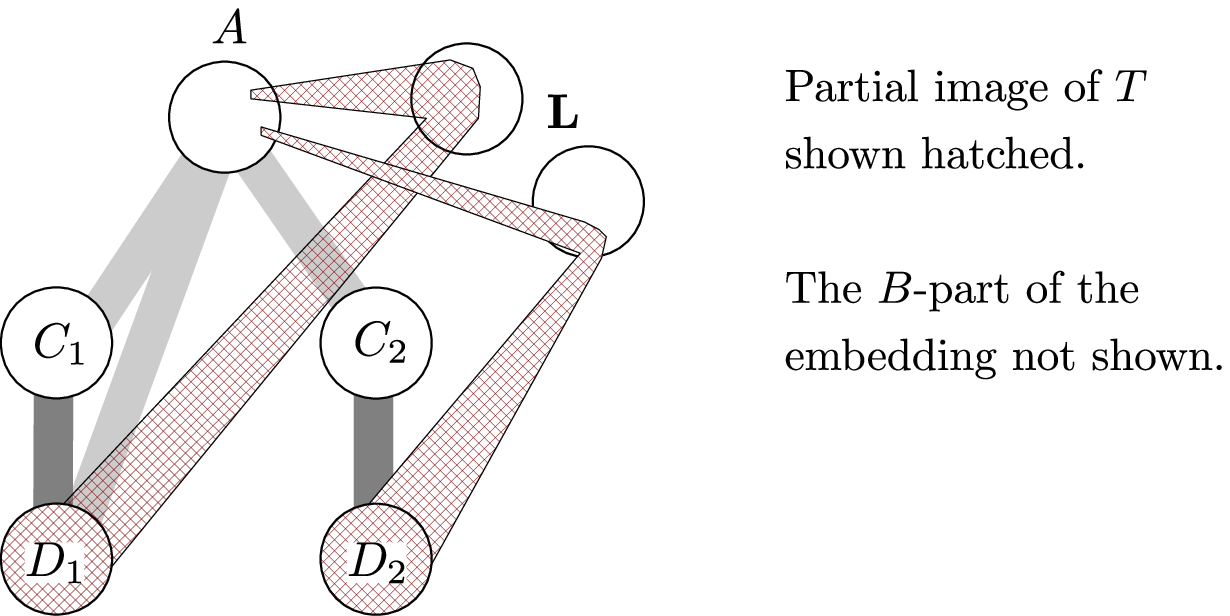}
\caption{The reason for requiring~\eqref{eq:additional} in the setting of \textbf{(H2)}. Consider two edges $C_1D_1,C_2D_2\in M$ such that only $C_2D_2$ satisfies~\eqref{eq:additional}. At some point during the embedding of~$T$, we may need to use the high-degree property of clusters in~$\mathbf{L}$. When doing so we cannot guarantee that we will fill the edges of $M$ in an efficient way. That is, we may end up filling~$D_1$ and~$D_2$ completely and leaving~$C_1$ and~$C_2$ untouched. If this happens, both regular pairs $C_1D_1$ and $C_2D_2$ are useless for embedding further shrubs. 
The used space in $C_2D_2$ equals the degree from~$A$ to $C_2D_2$. That is, we do not expect to embed anything more in the edge $C_2D_2$. The condition $\deg_{\BGblack}(A, \BL\cup V(M))\ge K$ ensures that we find free space somewhere else in the cluster graph to complete our embedding.
Clearly, the pair $C_1D_1$ does not have this favourable feature: the number of vertices used by the embedding is only half the degree of $A$ to $C_1D_1$.  In this case, the condition $\deg_{\BGblack}(A, \BL\cup V(M))\ge K$ is not strong enough.
\\
We do not need a counterpart of~\eqref{eq:additional} for $\neighbour_{\BGblack}(B)$. The reason is that we can schedule our embedding process in such a way that when we use the high-degree property of $\mathbf{L}$ we have already exhausted the degree from $B$ to $M$.}
\label{fig:why52}
\end{figure}
So far, we have not explained the role of condition~\eqref{eq:additional}. We include a relatively detailed explanation in Figure~\ref{fig:why52}. However, this condition is independent of the rest of the argument, and it may be sufficient for the reader to take granted that~\eqref{eq:additional} is crucial for the embedding to work.

\bigskip
We now try to give an analogue to Informal Lemma~\ref{lem:StructPigSt} in the sparse setting when the structure of $G$ is encoded in the sparse decomposition rather than in the cluster graph. Recall that in the dense setting sets suitable for embedding shrubs were clusters of a regular matching (that is, $M$), and clusters of large degree (that is, $\BL$). In the sparse setting, in addition to using a suitable matching of regular pairs $\mathcal M$ and large degree vertices $\largevertices{k}{\eta}{G}$ we can make use of two further sets for embedding shrubs: $V(\Gexp)$ (as explained in \cite[Section~\ref{p0.sssec:whyGexp}]{cite:LKS-cut0}) and the set $\smallatoms$ (as explained in \cite[Section~\ref{p0.sssec:whyavoiding}]{cite:LKS-cut0}). Thus, the counterpart of clusters $A_1$, $A_2$ and $A$ from Informal Lemma~\ref{lem:StructPigSt} is the set $\XA$ of vertices, which have degree at least $k$ into the set $\largevertices{k}{\alpha}{G}\cup V(\mathcal M)\cup V(\Gexp)\cup \smallatoms$.\footnote{The rather different looking formal definition of $\XA$ is given in~\eqref{eq:XAXBXC}. Below, we give an explanation for this difference.} Likewise, the counterpart of cluster $B$ in Informal Lemma~\ref{lem:StructPigSt} are vertices of $\XB$, which have degree at least $k/2$ into $\largevertices{k}{\alpha}{G}\cup V(\mathcal M)\cup V(\Gexp)\cup \smallatoms$.\footnote{The formal definition of $\XB$ is given again in~\eqref{eq:XAXBXC}.} We see that a sparse counterpart to $\mathbf{(H1)}$ would be two disjoint well-connected sets $\XA_1,\XA_2\subset \XA$. In Lemma~\ref{prop:LKSstruct} we achieve this in one of two possible ways. One way is finding a large \semiregular matching $\Mgood$ inside $\XA$; one can then take $\XA_1=V_1(\mathcal \Mgood)$ and $\XA_2=V_2(\mathcal \Mgood)$. This corresponds to $\mathbf{(K2)}$ in Lemma~\ref{prop:LKSstruct}\eqref{Mgoodisblack}. Next, suppose that $\XA$ induces sufficiently many edges. Then we take $\XA_1$ and $\XA_2$ to be a bipartition of $\XA$ corresponding to a maximum cut. Hence, the sets $\XA_1$ and $\XA_2$ are again well-connected. 
This corresponds to the case $e(\XA)\ge \eta kn/12$ in $\mathbf{(K1)}$ in Lemma~\ref{prop:LKSstruct}\eqref{Mgoodisblack}.
Similarly, if the sets $\XA$ and $\XB$ are well-connected, we are on a good track to getting a  sparse version of $\mathbf{(H2)}$. 

It remains to translate  condition~\eqref{eq:additional}. The right counterpart to this condition is
\begin{equation}\label{eq:additionalsparse}
\text{for every $XY\in \mathcal M$,}\quad 
\neighbour_{\Gcapt}(\XA)\cap X=\emptyset
\quad
\text{or}
\quad
\neighbour_{\Gcapt}(\XA)\cap Y=\emptyset\;.
\end{equation}

The actual statement of Lemma~\ref{prop:LKSstruct} deviates quite substantially from the informal account given above. So, let us now state an informal version of Lemma~\ref{prop:LKSstruct}. After that, we explain how it relates to the description above. Also, we mark the correspondence between this informal version and the actual lemma by using the same numbering. In particular, assertions \eqref{eq:M1}, \eqref{newpropertyS6}, \eqref{nicDoNAtom} in Lemma~\ref{prop:LKSstruct} are needed for reasons that cannot be explained in this high-level overview and are not reflected in the informal version. Further, statement of (\ref{eq:Mspots}') of our informal lemma carries only half of the information compared to the full version in Lemma~\ref{prop:LKSstruct}.

Let us now give the actual definitions of the sets $\XA$, $\XB$. Later, we explain how these definitions imply the features described above. 
\begin{definition}\label{def:XAXBXC}
Suppose that $k\in \NN$, $\gamma,\eta,\epsilon,\epsilon',\nu,\rho>0$ and $\Lambda,\Omega^{*},\Omega^{**}>0$.
Suppose that $G$ is a graph with a $(k,\Omega^{**},\Omega^*,\Lambda,\gamma,\epsilon,\nu,\rho)$-sparse decomposition 
$$\class=(\HugeVertices, \clusters,\DenseSpots, \Gblack, \Gexp,\smallatoms
)\;$$ with respect to $\largevertices{\eta}{k}{G}$ and
$\smallvertices{\eta}{k}{G}$.
Suppose further that  $\mathcal M_A,\mathcal M_B$ are
\semiregular matchings  in~$G$. We then define the triple \index{mathsymbols}{*XA@$\XA(\eta,\class, \mathcal M_A,\mathcal M_B)$} \index{mathsymbols}{*XB@$\XB(\eta,\class, \mathcal M_A,\mathcal M_B)$}
\index{mathsymbols}{*XC@$\XC(\eta,\class, \mathcal M_A,\mathcal
	M_B)$}
$(\XA,\XB,\XC)=(\XA,\XB,\XC)(\eta,\class, \mathcal M_A,\mathcal M_B)$
by setting
\begin{align}
\begin{split}\label{eq:XAXBXC}
\XA&:=\largevertices{\eta}{k}{G}\setminus V(\mathcal M_B)\;,\\
\XB&:=\left\{v\in V(\mathcal M_B)\cap \largevertices{\eta}{k}{G}\::\:\widehat{\deg}(v)<(1+\eta)\frac
k2\right\}\;,\\
\XC&:=\largevertices{\eta}{k}{G}\setminus(\XA\cup\XB)\;,
\end{split}
\end{align}
where 
$\widehat{\deg}(v)$ on the second
line is defined by
\begin{equation}\label{eq:defhatdeg}
\widehat{\deg}(v):=\deg_G\big(v,\smallvertices{\eta}{k}{G}\setminus
(V(\Gexp)\cup\smallatoms\cup V(\mathcal M_A\cup\mathcal M_B)\big)\;.
\end{equation}
\end{definition}
It is enough to restrict ourselves for the proof to the class $\LKSsmallgraphs{n}{k}{\eta}\subset \LKSgraphs{n}{k}{\eta}$. We intentionally leave out (or simplify) almost all numerical parameters in this informal statement.
\theoremstyle{theorem}
\newtheorem*{informalLKSStruc}{Informal version of Lemma~\ref{prop:LKSstruct}}
\begin{informalLKSStruc}
Suppose $\class=(\HugeVertices, \clusters,\DenseSpots, \Gblack, \Gexp,\smallatoms )$
is a sparse
decomposition of a graph $G\in\LKSsmallgraphs{n}{k}{\eta}$. We write $S^0:=\smallvertices{\eta}{k}{G}\setminus (V(\Gexp)\cup\smallatoms)$. Then there exist \semiregular matchings $\M_A$ and $\M_B$, such that for the sets $\XA$ and $\XB$ defined as in Definition~\ref{def:XAXBXC} we have
\begin{enumerate}
  \item[\eqref{prop6.1a}]
  $V(\mathcal M_A)\cap V(\mathcal
  M_B)=\emptyset$,
  \item[\eqref{eq:lastminute}] $V_1(\M_B)\subset S^0$,
  \item[$(\mathrm{\ref{eq:Mspots}}')$] for each $X\in \V(M_A)\cup\V(\M_B)$ we have that $X\subset \smallvertices{\eta}{k}{G}$ or
  $X\subset \largevertices{\eta}{k}{G}$,
  \item[\eqref{fewfewfew}] $e\big(\XA,S^0\setminus V(\mathcal M_A)\big)=0$,
\item[\eqref{Mgoodisblack}] if $\XA$ induces almost no edges and does not contain a substantial \semiregular matching\footnote{The exact quantification of ``almost no edges'' and ``substantial \semiregular matching'' does not in guarantee the former property to imply the latter. See also Remark~\ref{rem:K1vsK2}} then there is a substantial amount of edges between $\XA$ and $\XB$.
\end{enumerate}
\end{informalLKSStruc}
The \semiregular matching $\M_A\cup \M_B$ from the lemma plays the role of $\M$ in the motivation above.
It remains to justify the dissimilarities between the statement of the lemma and the text above. The first discrepancy is that the definitions of the sets $\XA$ and $\XB$ in~\eqref{eq:XAXBXC} are quite different from the ones in the motivation above. The other discrepancy is a seeming absence of~\eqref{eq:additionalsparse} in the statement. As for the first issue, consider an arbitrary vertex $v\in \XA$. Property~\eqref{fewfewfew} tells us that $v$ sends no edges to $S^0\setminus V(\M_A)\supset S^0\setminus (V(\M_A)\cup V(\M_B))$. As $v\in\largevertices{\eta}{k}{G}$, we have that $\deg(v,\largevertices{\eta}{k}{G}\cup V(\M_A)\cup V(\M_B)\cup V(\Gexp)\cup\smallatoms)\ge (1+\eta)k$, as needed. Next, consider a vertex $v\in \XB$. The fact that $v\in \largevertices{\eta}{k}{G}$ together with the definition of $\widehat{\deg}$ immediately gives that $\deg(v,\largevertices{\eta}{k}{G}\cup V(\M_A)\cup V(\M_B)\cup V(\Gexp)\cup\smallatoms)> (1+\eta)\frac k2$, again as needed.

Let us now turn to deriving~\eqref{eq:additionalsparse}. This property is required only for the counterpart of $\mathbf{(H2)}$. So, we can assume that we do not have the counterpart of $\mathbf{(H1)}$, that is, the set $\XA$ induces (almost) no edges. Let us now consider an arbitrary regular pair $(X,Y)$ in $\M_A\cup \M_B$. First assume that $(X,Y)\in\M_B$. Then~\eqref{eq:lastminute} tells us that $X\subset S^0$. We then have $\neighbour(\XA)\cap X=\emptyset$ by Property~\eqref{fewfewfew}, as needed for~\eqref{eq:additionalsparse}. Next, assume that $(X,Y)\in\M_A$. Then  Definition~\ref{def:LKSsmall}\eqref{en:LKSsmall.noSS} (together with (\ref{eq:Mspots}') of our informal lemma) tells us that at least one of  $X$ and $Y$ is contained in $\largevertices{\eta}{k}{G}$. Say this is $X$. We then have $X\subset \XA$. But the absence of edges inside $\XA$ tells us that $e(X,\XA)=0$, again as needed for~\eqref{eq:additionalsparse}.

\subsubsection{Rough versus fine structure}\label{sssec:roughversusfine}
In the dense case~\cite{PS07+} we can proceed with embedding $T$ using the regularity method immediately after having established a statement like Informal Lemma~\ref{lem:StructPigSt}. That is, we can zigzag-embed consecutive cut vertices $W_A\cup W_B$ of $T$ in $AB$, or $A_1A_2$. When we arrive at a shrub $t\in\shrubA\cup\shrubB$ stemming from cut vertex $u\in W_A\cup W_B$ embedded to a cluster $D$ (that is, $D=A$, $D=B$, $D=A_1$, or $D=A_2$) we can use  $\mathbf{(H1)}$ or $\mathbf{(H2)}$ to find an edge $XY\in E(\BGblack)$ such that $DX\in E(\BGblack)$ and the pair $(X,Y)$ has not been filled-up. Then, \emph{(i)} using that $DX\in E(\BGblack)$ we traverse from $D$ to $XY$, \emph{(ii)} we embed $t$ in $(X,Y)$, and \emph{(iii)} if $t$ is an internal shrub, we again use that $XD\in E(\BGblack)$ to traverse back\footnote{As said at the beginning of Section~\ref{ssec:motivation}, if $t$ is internal, then both of its neighboring cut vertices are in $W_A$. In particular, the distance between these two cut vertices is even. That means that to traverse back to $D$, we really use the pair $XD\in E(\BGblack)$.} to $D$ and continue embedding cut vertices in $AB$ or $A_1A_2$.

In the current setting of the sparse decomposition, the structure given by Lemma~\ref{prop:LKSstruct} would allow us to carry out counterparts to \emph{(i)} and \emph{(ii)} (even though there is a number of technical obstacles). That is, we would be able to embed consecutive cut vertices, to traverse to locations suitable for shrubs and to embed these shrubs. However, carrying a counterpart to \emph{(iii)} is a major problem. The symmetry-based argument from the dense case ``if $DX$ is an edge then $XD$ is an edge and thus we can traverse in both directions'' does not work 
when the shrub is not to be embedded in a cluster, but in a subset of $\smallatoms$ or $\Gexp$.
This is going to be addressed in~\cite{cite:LKS-cut2}, where we clean the rough structure in such a way that it will allow a counterpart to \emph{(iii)}.

\subsection{The role of Lemma~\ref{lem:Separate} in the proof of
	Lemma~\ref{prop:LKSstruct}}
\label{ssec:whyaugment}
In this section, we explain the role of Lemma~\ref{lem:Separate} in our proof of
Lemma~\ref{prop:LKSstruct}. That is, we want to explain why it is not possible in general to find regular matchings $\M_A$ and $\M_B$ from the informal version of Lemma~\ref{prop:LKSstruct} inside the cluster graph $\BGblack$. Because of this we will have to find a suitable ``re-regularization'' which turns out to be provided by Lemma~\ref{lem:Separate}.

Recall the motivation from Section~\ref{ssec:motivation}. We wish to find two sets $\XA$ and $\XB$ (or two sets within $\XA$) which are suitable for
embedding the cut vertices $W_A$ and $W_B$ of a $(\tau k)$-fine partition
$(W_A,W_B,\shrubA,\shrubB)$ of $T$.
In this sketch we just focus on finding $\XA$; the ideas behind finding
a suitable set $\XB$ are similar. To accommodate all the shrubs from $\shrubA$ --- which might contain up to $k-2$
vertices in total --- we need $\XA$ to have total degree at least $k$ into  the sets $\largevertices{\eta}{k}{G}$, $V(\Gexp)$, $\smallatoms$, together with vertices of any fixed matching $\M$
consisting of regular pairs. 
This motivates us to look for a
\semiregular matching~$\M$ which covers as much as possible of the set
$S^0:=\smallvertices{\eta}{k}{G}\setminus
\left(V(\Gexp)\cup\smallatoms\right)$. as these are the vertices that are not
utilizable otherwise.
As a next step one would prove
that there is a set $\XA$ with $$\mindeg\left(\XA,V(G)\setminus (S^0\setminus
V(\M))\right)\gtrsim k\;.$$ 
(By $\gtrsim k$ we mean larger than $k$ by a suitable small additional approximation factor.)

In the dense setting~\cite{PS07+}, where the
structure of $G$ is determined by $\BGblack$, and
where $S^0=\smallvertices{\eta}{k}{G}$, such a matching $\M$ can be found inside
$\BGblack$ using the Gallai--Edmonds Matching Theorem. But here, just working
with $\BGblack$ is not enough for finding a suitable \semiregular matching as the
following example shows.

\begin{figure}[ht] \centering \includegraphics[scale=0.8]{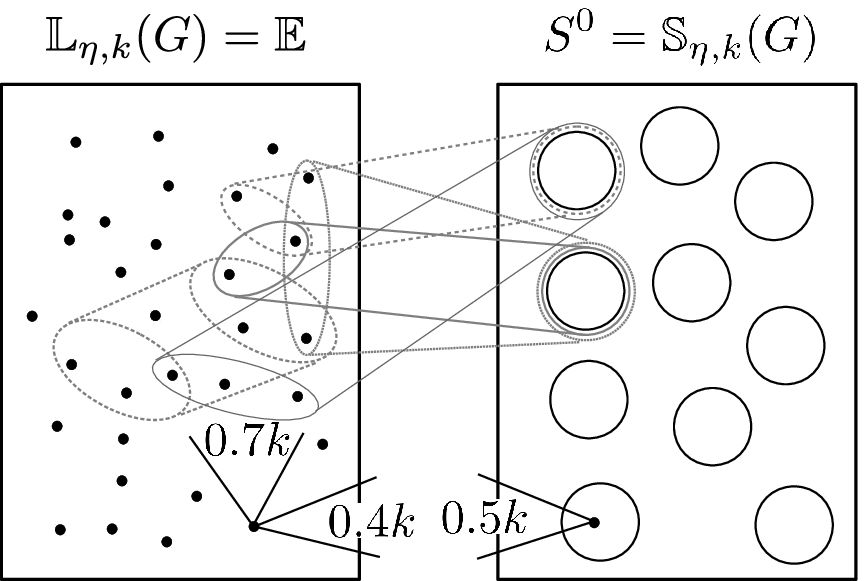}
	\caption[Example of graph with $\BGblack$ empty]{An example of a graph $G\in\LKSgraphs{n}{k}{\eta:=\frac{1}{10}}$ in which $\BGblack$ is empty, yet there is no candidate set for $\XA$ of vertices which have degrees at least $k$ outside the set~$S^0$. Sample dense spots are shown in grey.}
	\label{fig:whyEnhancing}
\end{figure}
Figure~\ref{fig:whyEnhancing} shows a graph $G$ with
$\largevertices{\eta}{k}{G}= \smallatoms$. Let us describe the construction of such a graph~$G$. We partition the vertex sets into to-be sets $\smallvertices{\eta}{k}{G}$ and $\largevertices{\eta}{k}{G}$. We further gather vertices of $\smallvertices{\eta}{k}{G}$ into clusters. We now insert edges into $G$. All the edges inserted will be in the form of dense spots. These dense spots have either both parts in $\largevertices{\eta}{k}{G}$, or one part $\largevertices{\eta}{k}{G}$ and the other in $\smallvertices{\eta}{k}{G}$. We do this so that each inserted dense spot in the  $\smallvertices{\eta}{k}{G}$-part respects the cluster structure, while it behaves in a random-like way in the $\largevertices{\eta}{k}{G}$-part.  Further, we require that each $\largevertices{\eta}{k}{G}$-vertex sends $0.7k$ edges to
$\largevertices{\eta}{k}{G}$ and $0.4k$ edges to $\smallvertices{\eta}{k}{G}$,
and each $\smallvertices{\eta}{k}{G}$-vertex receives $0.5k$ edges
from $\largevertices{\eta}{k}{G}$. Clearly, such a construction is possible. 

The point of the construction is that $\smallatoms=\largevertices{\eta}{k}{G}$, and that
$S^0=\smallvertices{\eta}{k}{G}$ form clusters which do not induce any dense
regular pairs.  No vertex has degree $\gtrsim k$ outside $S^0$, and the
cluster graph~$\BGblack$ contains no matching.

However, in this situation we can still find a large \semiregular matching $\M$ between $\largevertices{\eta}{k}{G}$ and $\smallvertices{\eta}{k}{G}$, by regularizing the crossing dense spots $\DenseSpots$ (which we can assume to be the original dense spots inserted in our construction).
In general, obtaining a \semiregular matching is, of course, more complicated. Given the above example, one may ask whether the graph $\BGblack$ has any role at all. The answer is that for constructing $\M$, we can either use directly the edges of $\BGblack$, or, if we do not have these edges the information about their lack helps us to find $\M$ elsewhere.

\subsection{Finding the structure}\label{sec6first}
We can now state the main result of this paper.
\begin{proposition}\label{prop:LKSstruct}
For every $\eta\in (0,1)$, $\Omega>0, \gamma\in (0,\eta/3)$ there is a number
$\beta>0$ so that for every $\epsilon\in(0,\frac{\gamma^2\eta}{12})$ there exist $\epsilon',\pi>0$ such that for every $\nu>0$ there exists a number $k_0\in \NN$ such that for every $\Omega^*$ with  $\Omega^*< \Omega$, every $\Omega^{**}$ with $\Omega^{**}>\max\{2,\Omega^*\}$ and every $k$ with $k> k_0$  the following holds.

Suppose $\class=(\HugeVertices, \clusters,\DenseSpots, \Gblack, \Gexp,\smallatoms )$
is a $(k,\Omega^{**},\Omega^*,\Lambda,\gamma,\epsilon',\nu,\rho)$-sparse
decomposition of a graph $G\in\LKSsmallgraphs{n}{k}{\eta}$ with respect to  $S:=\smallvertices{\eta}{k}{G}$ and
$L:=\largevertices{\eta}{k}{G}$ which captures all but at most $\eta kn/6$ edges
of $G$. Let $\clustersize$ be the size of the clusters $\clusters$.\footnote{The number $\clustersize$ is irrelevant when $\clusters=\emptyset$. In particular, note that in that case we necessarily have $\M_A=\M_B=\emptyset$ for the \semiregular matchings given by the lemma.} Write 
\begin{equation}\label{eq:S0alt}
S^0:=S\setminus \left(V(\Gexp)\cup\smallatoms\right)\;.
\end{equation}

Then $\GD$ contains  two
$(\epsilon,\beta,\pi \clustersize)$-\semiregular matchings
$\mathcal M_A$ and $\mathcal M_B$ such
that for the triple $(\XA,\XB,\XC):=(\XA,\XB,\XC)(\eta,\class, \mathcal M_A,\mathcal M_B)$
we have
\begin{enumerate}[(a)]
  \item \label{prop6.1a}
$V(\mathcal M_A)\cap V(\mathcal
M_B)=\emptyset$,
\item \label{eq:lastminute} $V_1(\M_B)\subset S^0$,
\item \label{eq:Mspots} for each $(X_1,X_2)\in\M_A\cup\M_B$, there is a dense spot  $(D_1,D_2; E_D)\in \DenseSpots$ with $X_1\subset
D_1$, $X_2\subset D_2$, and furthermore, either $X_1\subset S$ or $X_1\subset L$, and
$X_2\subset S$ or $X_2\subset L$,
\item \label{eq:M1} for each $X_1\in\V_1(\M_A\cup\M_B)$ there exists a cluster $C_1\in \clusters$ such that $X_1\subset C_1$, and for each $X_2\in\V_2(\M_A\cup\M_B)$ we have $X_2\subset L\cap \smallatoms$ or  there exists $C_2\in \clusters$ such that $X_2\subset C_2$,
\item \label{fewfewfew} $e_{\Gcapt}\big(\XA,S^0\setminus V(\mathcal M_A)\big)\le
\gamma kn$,
\item\label{newpropertyS6} $e_{\Gblack}(V(G)\setminus
V(\M_A\cup \M_B))\le \epsilon \Omega^* kn$,
\item\label{nicDoNAtom} for the \semiregular matching $\NAtom:=\{(X_1,X_2)\in\M_A\cup\M_B\::\: (X_1\cup X_2)\cap\smallatoms\not=\emptyset\}$ we have $e_{\Gblack}\big(V(G)\setminus
V(\M_A\cup \M_B),V(\NAtom)\big)\le \epsilon \Omega^* kn$,
\item\label{Mgoodisblack} for
$\Mgood:=\{(X_1,X_2)\in \mathcal M_A\::\: X_1\cup X_2\subset \XA \}$ we have that each $\Mgood$-edge is an edge of
$\BGblack$,
and  at least one of the following conditions holds 
\begin{itemize}
\item[{\bf(K1)}]$2e_G(\XA)+e_G(\XA, \XB)\ge \eta kn/3$,
\item[{\bf(K2)}]  $|V(\Mgood)|\ge \eta n/3$.
\end{itemize}
\end{enumerate}
\end{proposition}

\begin{remark}\label{rem:K1vsK2}
As explained in Section~\ref{ssec:motivation}, property~\eqref{Mgoodisblack} is the most important part of Lemma~\ref{prop:LKSstruct}.
Note that the assertion {\bf(K2)} implies a
quantitatively weaker version of {\bf(K1)}. Indeed, consider $(X_1,X_2)\in \M_A$. An
average vertex $v\in X_1$ sends at least $\beta\cdot \pi\clustersize\ge \beta\cdot \pi\nu k$ 
edges to
$X_2$. Thus, if $|V(\Mgood)|\ge \eta n/3$ then
$\Mgood$ induces at least $(\eta n/6)\cdot\beta\cdot
\pi\nu k=\Theta(kn)$ edges in $\XA$. Such a bound, however, would be insufficient
for our purposes as later $\eta\gg\pi,\nu$. So, the deficit in the number of edges asserted in {\bf(K2)} (compared to the $e_G(\XA)\ge \eta kn/12$ part of {\bf(K1)}) is compensated by the fact that these edges are already regularized.
\end{remark}

For  the proof we need the well-known Gallai--Edmonds Matching Theorem, which we state next.
A graph $H$ is called \index{general}{factor-critical}{\em
factor-critical} if  $H-v$ has a
perfect matching for each $v\in V(H)$.

\begin{theorem}[Gallai--Edmonds matching theorem (see for instance~{\cite[Theorem 2.2.3]{Die05}})]
\label{thm:GallaiEdmonds} Let $H$ be a graph. Then there exist a set
$Q\subset V(H)$ and a matching $M$ of size $|Q|$ in $H$ such that
\begin{enumerate}[(1)]
\item every component of $H-Q$ is factor-critical, and
\item $M$
matches every vertex in $Q$ to a different component of $H-Q$.
\end{enumerate}
\end{theorem}
The set $Q$ in Theorem~\ref{thm:GallaiEdmonds} is often referred to as a \index{general}{separator}{\em separator}.

\def\LPJJ{_\PARAMETERPASSING{L}{lem:Separate}}
\begin{proof}[Proof of Lemma~\ref{prop:LKSstruct}]
\addtocounter{theorem}{-1}
The idea of the proof is to first obtain some information about the structure of the graph $\BGblack$ with the help of Theorem~\ref{thm:GallaiEdmonds}.
 Then the structure given by Theorem~\ref{thm:GallaiEdmonds} is refined by Lemma~\ref{lem:Separate} to yield the assertions of the lemma.

Let us begin with setting the parameters. 
Let $\beta:=\beta\LPJJ$ be given by Lemma~\ref{lem:Separate} for input parameters $\Omega\LPJJ:=\Omega$, $\rho\LPJJ:=\gamma^2$,
and let $\epsilon'$ and $\pi$ be given by Lemma~\ref{lem:Separate} 
for further input parameter $\epsilon\LPJJ:=\epsilon$. Last, let $k_0$ be given by Lemma~\ref{lem:Separate} with the above parameters and $\gamma\LPJJ:=\nu$.

Without loss of generality we assume
that $\epsilon'\le \epsilon$ and $\beta<\gamma^2$.
 We write $\BS :=\{C\in\clusters\::\: C\subset S\}$ and
$\BL:=\{C\in\clusters\::\: C\subset L\}$. Further, let $\BSN:=\{C\in\BS\::\: C\subset S^0\}$. 
 
Let $\SEPARATOR$ be a separator and let $N_0$ be a matching given by
Theorem~\ref{thm:GallaiEdmonds} applied to  the graph $\BGblack$. We will presume that the pair $(\SEPARATOR,N_0)$ is chosen among all
the possible choices so that the number of vertices of $\BSN$ that are isolated
in $\BGblack-\SEPARATOR$ and are not covered by
$N_0$ is minimized. Let $\BSI$ denote the set of
vertices in $\BSN$ that are isolated in $\BGblack-\SEPARATOR$. Recall that the components of $\BGblack-\SEPARATOR$ are factor-critical.
 
Define
 $\BSR\subset V(\BGblack)$ as a minimal set such that
\begin{itemize}
\item $\BSI\sm V(N_0) \subseteq \BSR$, and
\item if $C\in \BS$  and there is an edge $DZ\in
E(\BGblack)$ with $Z\in \BSR$, $D\in
\SEPARATOR$, $CD\in N_0$ then $C\in \BSR$.
\end{itemize}

Then each vertex from $\BSR$ is
reachable from $\BSI\sm V(N_0)$ by 
 a path  in $\BGblack$ that alternates between $\BSR$ and $\SEPARATOR$, and has
 every second edge in $N_0$. 
Also note that for all $CD\in N_0$ with $C\in\SEPARATOR$ and $D\in\BSN\setminus\BSR$ we have 
\begin{equation}\label{anti-skub}
\deg_{\BGblack}(C,\BSR)=0\;.
\end{equation}

Let us prove another property of $\BSR$.
\begin{claim}\label{cl:SRS0}
$\BSR\subseteq \BSI \subseteq \BSR\cup V(N_0)$. In particular, $\BSR\subset \BSN$.
\end{claim}
\begin{proof}[Proof of Claim~\ref{cl:SRS0}]
By the definition of $\BSR$, we only need to show that $\BSR\subset\BSI$.
So suppose there is  a vertex $C\in \BSR\sm\BSI$. 
By the definition of $\BSR$ there is a non-trivial path $R$ going from $\BSI\sm
V(N_0)$ to $C$ that alternates between $\BSR$ and $\SEPARATOR$, and has every second edge in
$N_0$. Then, the matching $N_0':=N_0\triangle E(R)$ covers more vertices of
$\BSI$ than $N_0$ does. Further, it is straightforward to check that the separator $\SEPARATOR$
together with the matching $N_0'$ satisfies the assertions of
Theorem~\ref{thm:GallaiEdmonds}. This is a contradiction, as desired.
\end{proof}

Using a very similar alternating path argument we see the
following.
\begin{claim}\label{augmiL}
If $CD\in N_0$ with $C\in \SEPARATOR$ and $D\notin \BSI$ then
$\deg_{\BGblack}(C,\BSR)=0$.
\end{claim}
\HIDDENPROOF{ Indeed, suppose
there is $CD\in N_0$ with $C\in \SEPARATOR$ and $D\notin \BSI$ but $CC'\in E(\BGblack)$ for some $C'\in\BSR$. Then there is a
non-trivial path $R$ with its first
three vertices $D,C$ and $C'$ and its last vertex in $\BSI\sm V(N_0)$ that alternates between $\BSR$ and $\SEPARATOR$, and has every second edge
in $N_0$. Then, the matching $N_0':=N_0\triangle E(R)$ covers
more vertices of $\BSI$ than $N_0$ does. Further, the
separator $\SEPARATOR$ together with the matching $N_0'$ satisfies the
assertions of Theorem~\ref{thm:GallaiEdmonds}. This
contradiction proves Claim~\ref{augmiL}.}

Using the factor-criticality of the components of $\BGblack-\SEPARATOR$ we extend $N_0$ to a matching $N_1$ as follows. For each component $K$ of
$\BGblack-\SEPARATOR$ which meets $V(N_0)$, we add a perfect matching of
$K-V(N_0)$. Furthermore, for each non-singleton component
$K$ of $\BGblack-\SEPARATOR$ which does not meet
$V(N_0)$, we add a matching which meets all but exactly one
vertex of $\BL\cap V(K)$. 
This is possible as by the
definition of the class $\LKSsmallgraphs{n}{k}{\eta}$ we have
that $\BGblack-\BL$ is edgeless, and so
$\BL\cap V(K)\neq\emptyset$.
This choice of $N_1$ guarantees that
\begin{equation}\label{eq:bir1}
e_{\BGblack}(\clusters\setminus V(N_1))=0\;.
\end{equation}

We set
\begin{align*}
M&:=\left\{ C_1C_2\in N_0\::\: C_1\in \BSR, C_2\in \SEPARATOR\right\}\;.
\end{align*}
We have that
\begin{equation}\label{eq:645}
e_{\BGblack}\big(\clusters\setminus V(N_1),V(M)\cap\BSR\big)=0\;.
\end{equation}

As $\BS$ is an independent set in $\BGblack$, we have that
\begin{equation}\label{eq:BinL}
\SEPARATOR_M:=V(M)\cap \SEPARATOR\subset \BL\;.
\end{equation}

The matching $M$ in $\BGblack$ corresponds to an $(\epsilon',\gamma^2,\clustersize)$-\semiregular matching $\mathcal M$ in the underlying graph $\Gblack$, with $V_2(\mathcal M)\subset \bigcup \SEPARATOR$ (recall that \semiregular matchings have orientations on their edges). Likewise, we define $\mathcal N_1$ as the $(\epsilon',\gamma^2,\clustersize)$-\semiregular matching corresponding to $N_1$. The $\mathcal N_1$-edges are oriented so that $V_1(\mathcal N_1)\cap \bigcup \SEPARATOR=\emptyset$; this condition does not specify orientations of all the $\mathcal N_1$-edges and we orient the remaining ones in an arbitrary fashion. We write $\SR:=\bigcup\BSR$.

\begin{claim}\label{cl:CrossEdges1}
$e_{\Gcapt}\big(L\setminus (\smallatoms\cup V(\mathcal M)),\SR\big)=0$.
\end{claim}

\begin{proof}[Proof of Claim~\ref{cl:CrossEdges1}] We start by
showing that for every cluster $C\in \BL\setminus V(M)$ we have
\begin{equation}\label{eq:degCSR}
\deg_{\BGblack}(C,\BSR)=0\;.
\end{equation} 
First, if $C\not
\in \SEPARATOR$, then~\eqref{eq:degCSR} is true since $\BSR\subseteq \BSI$ by
Claim~\ref{cl:SRS0}. So suppose that $C\in \SEPARATOR$, and let $D\in V(\BGblack)$ be such that $DC\in N_0$. Now if  $D\notin\BSI$
then~\eqref{eq:degCSR} follows from Claim~\ref{augmiL}. 
On the other hand, suppose $D\in\BSI\subseteq
\BS^0$. As $C\notin V(M)$, we know that $D\notin\BSR$, and
thus,~\eqref{eq:degCSR}
follows from~\eqref{anti-skub}. 

Now, by~\eqref{eq:degCSR}, $\Gblack$  has no edges between $L\setminus (\smallatoms\cup
V(\mathcal M))$ and $\SR$. Also, no such edges can be in
$\Gexp$  or incident with $\smallatoms$, since
$\BSR\subseteq \BSN$ by Claim~\ref{cl:SRS0}. Finally, since $G\in\LKSsmallgraphs{n}{k}{\eta}$ and $\Omega^{**}>2>(1+\eta)$, there
are no edges between $\HugeVertices$ and $S$. This proves the claim.
\end{proof}

We prepare ourselves for an application of Lemma~\ref{lem:Separate}. The
numerical parameters of the lemma are $\Omega\LPJJ,\rho\LPJJ,\epsilon\LPJJ$ and $\gamma\LPJJ$ as above. The input objects for the lemma are the graph $\GD$ of
order $n'\le n$, the collection of $(\gamma k,\gamma)$-dense spots
$\DenseSpots$, the matching $\mathcal M$, the $(\nu k)$-ensemble $\mathcal
C_\PARAMETERPASSING{L}{lem:Separate}:=\BSR\sm V(N_1)$, and the set $Y_\PARAMETERPASSING{L}{lem:Separate}:=L\cap
\smallatoms$.
Note that Definition~\ref{bclassdef}, item~\ref{defBC:dveapul}, implies that $\DenseSpots$ absorbs $\mathcal M$. Further,~\eqref{eq:510ass1} is satisfied by Definition~\ref{bclassdef}, item~\ref{defBC:prepartition}.

The output of Lemma~\ref{lem:Separate} is an $(\epsilon,\beta,\pi \clustersize)$-\semiregular matching $\mathcal M'$ with the following properties. 
\begin{enumerate}[(I)]
\item\label{eq:MsubMApp}
$|V(\mathcal M)\setminus V(\mathcal M')|<\epsilon n'\le \epsilon n$.
\item\label{theCandDwelike}
For each $(X_1,X_2)\in\mathcal M'$ there are sets  $C\in \BSR$ 
and $(D_1,D_2;E_D)\in\DenseSpots$
such that $X_1\subseteq C\cap D_1$ and either $X_2\subseteq L\cap \smallatoms\cap D_2$ or there exists $C'\in\SEPARATOR_M$ such that $X_2\subseteq L\cap \smallatoms\cap C'$. 

(Indeed, to see this we use that $\V_1(\mathcal
M)\subset \BSR$ and that $V_2(\mathcal M)\subseteq \bigcup \SEPARATOR_M$ by the
definition of~$\mathcal M$.)
\item\label{eq:SeparatedM} There is a partition of $\mathcal M'$ into $\mathcal M_1$ and $ \mathcal M_B$ such that
$$e_{\GD}\left(
\:
\left(\left(\SR\setminus V(\mathcal N_1)\right)\cup V_1(\mathcal M)\right)\setminus V_1(\mathcal M_1)
\:,\:
\left((L\cap \smallatoms)\cup V_2(\mathcal M)\right)\setminus V_2(\mathcal M_B)
\:\right)<\gamma k n' \;.$$
\end{enumerate}
We claim that also
\begin{enumerate}[(I)]\setcounter{enumi}{3}
\item\label{eq:M'Next}
$V(\mathcal{M'})\cap V(\mathcal N_1\setminus \mathcal M)=\emptyset$.
\end{enumerate}
Indeed, let $(X_1,X_2)\in\mathcal M'$
be arbitrary. Then by~\eqref{theCandDwelike} there is $C\in
\BSR$ such that $X_1\subset C$. By Claim~\ref{cl:SRS0},
$C$ is a singleton component of $\BGblack-\SEPARATOR$. In particular, if $C$ is covered by $N_1$
then $C\in V(M)$. It follows that $X_1\cap V(\mathcal N_1\setminus \mathcal M)=\emptyset$.
In a similar spirit, the easy fact that $(Y\cup\bigcup \SEPARATOR_M)\cap V(\mathcal N_1\setminus \mathcal M)=\emptyset$ together with~\eqref{theCandDwelike} gives $X_2\cap V(\mathcal N_1\setminus \mathcal M)=\emptyset$. This establishes~\eqref{eq:M'Next}.

\medskip

Observe that~\eqref{theCandDwelike} implies that $V_1(\mathcal M')\subset \SR$, and so,  by Claim~\ref{cl:SRS0} we know that 
\begin{equation}\label{eq:Asub}
V_1(\mathcal M_B)\subseteq \SR\subseteq \bigcup\BSI\subseteq\SN.
\end{equation}

Set
\begin{equation}\label{eq:defMA} 
\mathcal M_A:=(\mathcal N_1\setminus \mathcal M)\cup \mathcal M_1\;.
\end{equation}

Then $\mathcal M_A$ is an
$(\epsilon,\beta,\pi\clustersize)$-\semiregular
matching. Note that from now on, the sets $\XA, \XB$ and $\XC$ are defined. The
situtation is illustrated in Figure~\ref{fig:struktura1}.
\begin{figure}[t] \centering
\includegraphics[scale=1.0]{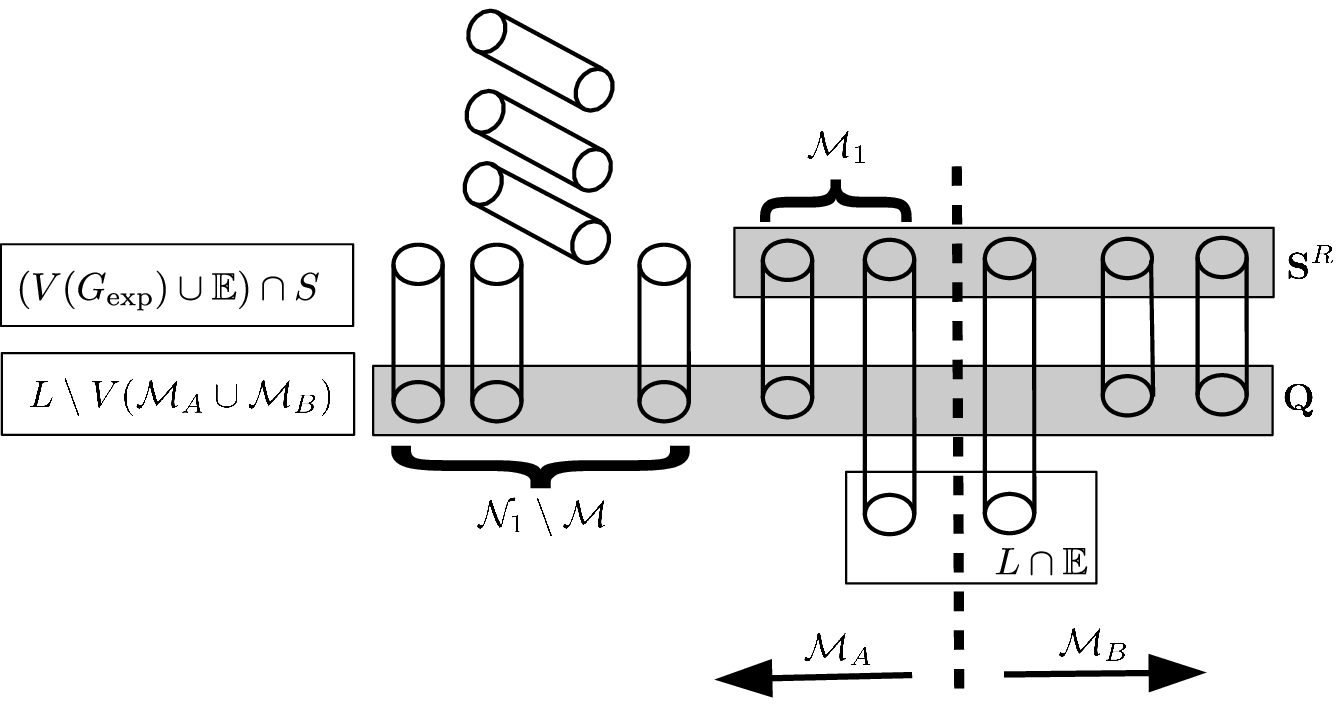}
\caption[Situation in $G$ in Lemma~\ref{prop:LKSstruct} after applying Lemma~\ref{lem:Separate}]{The situation in $G$ after applying Lemma~\ref{lem:Separate}. The dotted line illustrates the separation as in~(III).} 
\label{fig:struktura1}
\end{figure}
By~\eqref{eq:M'Next}, we
have $V(\mathcal M_A)\cap V(\mathcal M_B)=\emptyset$, as required for
Lemma~\ref{prop:LKSstruct}\eqref{prop6.1a}. Lemma~\ref{prop:LKSstruct}\eqref{eq:lastminute} follows from~\eqref{eq:Asub}. The claim below asserts that the next two properties are satisfied as well.
\begin{claim}\label{cl:L54cd}
Lemma~\ref{prop:LKSstruct}\eqref{eq:Mspots} and
Lemma~\ref{prop:LKSstruct}\eqref{eq:M1} are satisfied.
\end{claim}
\begin{proof}[Proof of Claim~\ref{cl:L54cd}]
Consider an arbitrary pair $(X_1,X_2)\in \M_A\cup \M_B$. Either we have that $(X_1,X_2)\in\mathcal N_1$ or $(X_1,X_2)\in\M'$. In the former case, $X_1X_2$ is an edge in $\BGblack$. Then the properties for $(X_1,X_2)$ asserted in Lemma~\ref{prop:LKSstruct}\eqref{eq:Mspots} and
Lemma~\ref{prop:LKSstruct}\eqref{eq:M1} follow from the fact that the cluster graph is prepartitioned with respect to $S$ and $L$, and from Definition~\ref{bclassdef}\eqref{defBC:dveapul}. 

In the case $(X_1,X_2)\in\M'$, the asserted properties are given by~\eqref{theCandDwelike}.
\end{proof}

We now turn to Lemma~\ref{prop:LKSstruct}\eqref{fewfewfew}. First we prove
some auxiliary statements. 

\begin{claim}\label{cl:S0mN1}
We have 
$\BSN\setminus V(N_1\setminus M)\subset \BSR$.
\end{claim}
\begin{proof}[Proof of Claim~\ref{cl:S0mN1}]
Let $C\in\BSN\setminus V(N_1\setminus M)$. 
Note that if $C\notin \BSI$, then $C\in V(N_1)$. On the other hand, if $C\in \BSI$, then we use  Claim~\ref{cl:SRS0} to see that $C\in\BSR\cup V(N_1)$.
We deduce that in either case $C\in \BSR\cup V(N_1)$. The choice of $C$ implies that $C\in \BSR\cup V(M)$.
Now, if $C\in V(M)$, then $C\in\BSR$, by~\eqref{eq:BinL} and by the definition of $M$. Thus $C\in\BSR$, as desired.
\end{proof}

It will be convenient to work with a set $\bar S^0\subset S^0$, $\bar S^0:=\left(S\cap\bigcup \clusters\right)\setminus V(\Gexp)=\bigcup \BSN$. 
The next two easy claims assert absence of edges of certain types incident to $S^0$ and $\bar S^0$.
\begin{claim}\label{cl:S0barS0iso}
The vertices in $S^0\setminus \bar S^0$ are isolated in $\Gcapt$.
\end{claim}
\begin{proof}[Proof of Claim~\ref{cl:S0barS0iso}]
Indeed, let us check Definition~\ref{capturededgesdef}. Clearly, $S^0\setminus \bar S^0$ is disjoint from $V(\Gblack)$ and $V(\Gexp)$. Further, $S^0\setminus \bar S^0$ sends no edges to $\HugeVertices$, by Definition~\ref{def:LKSsmall}\eqref{en:LKSsmall.noSS}. Lastly, the set $S^0\setminus \bar S^0$ is disjoint from the ``avoiding edges'' spanned by the vertex sets $\smallatoms$ and $\smallatoms\cup \bigcup \clusters$.
\end{proof}
\begin{claim}\label{cl:S0avoid}
We have $\Gcapt[L\cap \smallatoms,\bar S^0]=\GD[L\cap \smallatoms,\bar S^0]$.
\end{claim}
\begin{proof}[Proof of Claim~\ref{cl:S0avoid}]
The $\supset$-inclusion  of the edge-sets is clear. 

Next, recall that Definition~\ref{capturededgesdef} tells us that each edge in $\Gcapt$ between $\smallatoms$ and $\bigcup \clusters$ is either in $\Gexp$ or in $\GD$. As $\bar S^0\cap V(\Gexp)=\emptyset$, the $\subset$-inclusion follows.
\end{proof}

By Claim~\ref{cl:S0mN1}, we have
\begin{equation}\label{guanacos}
\bar S^0\setminus V(\M_A) \subset
\big(\bigcup \BSN\setminus V(N_1\setminus M)\big) \setminus V(\M_A) 
  \subset
\SR \setminus V(\M_A).
\end{equation}

 As every edge incident to
$S^0\setminus \bar S^0$ is uncaptured, we see that
\begin{align}
\label{fewfewfew:i}
E_{\Gcapt}\big(\XA\cap \smallatoms,S^0\setminus V(\mathcal
M_A)\big)
&= E_{\Gcapt}\big(\XA\cap \smallatoms,\bar S^0\setminus V(\mathcal
M_A)\big)\notag \\
\JUSTIFY{$\XA\cap \smallatoms=(L\cap \smallatoms)\setminus V(\mathcal M_B)$, C\ref{cl:S0avoid}}
& =
E_{\GD}\big((L\cap \smallatoms)\setminus V(\mathcal M_B),\bar S^0\setminus V(\mathcal
M_A)\big)\notag \\
\JUSTIFY{by \eqref{guanacos}}&\subset E_{\GD}\big(\:(L\cap \smallatoms)\setminus V(\mathcal
M_B)\:,\:\SR \setminus V(\M_A)\:\big).
\end{align}

\begin{claim}\label{cl:fewfewfew:ii}
We have
\begin{align*}
\label{fewfewfew:ii}
E_{\Gblack}\big(\XA\cap\bigcup\clusters,S^0\setminus V(\mathcal
M_A)\big)\subset  E_{\GD}\left(
\left((L\cap \smallatoms)\cup V_2(\mathcal M)\right)\setminus V_2(\mathcal M_B),
\SR \setminus V(\M_A)
\:\right).
\end{align*} 
\end{claim}
Before proving Claim~\ref{cl:fewfewfew:ii}, let us see that it implies Lemma~\ref{prop:LKSstruct}\eqref{fewfewfew}.
As
$G\in\LKSsmallgraphs{n}{k}{\eta}$, there are no edges between $\HugeVertices$
and $S$. That means that any captured edge from $\XA$ to $S^0\setminus V(\mathcal M_A)$ must start in 
$\smallatoms$ or in $\bigcup\clusters$, and must be contained in $\GD$.
Thus Lemma~\ref{prop:LKSstruct}\eqref{fewfewfew} follows by
plugging~\eqref{eq:SeparatedM} into~\eqref{fewfewfew:i}
and into Claim~\ref{cl:fewfewfew:ii}.

Let us now turn to proving Claim~\ref{cl:fewfewfew:ii}. 
\begin{proof}[Proof of Claim~\ref{cl:fewfewfew:ii}] First, observe that by the definition of $\XA$ and by the definition of $\mathcal M$ (and $M$) we have
\begin{equation}\label{eq:Hsplit}
 \XA\cap\bigcup\clusters\subset (V_2(\mathcal M)\setminus V_2(\mathcal M_B))\cup (L\setminus (\smallatoms\cup V(\mathcal M)))\;.
\end{equation}
Further, by applying~\eqref{guanacos} and Claim~\ref{cl:CrossEdges1} we get
\begin{equation}\label{eq:H0}
E_{\Gblack}\big(L\setminus (\smallatoms\cup V(\mathcal M)),\bar S^0\setminus V(\mathcal M_A)\big)=\emptyset\;.
\end{equation}

Therefore, we obtain
\begin{align*}
E_{\Gblack}\big(\XA\cap\bigcup\clusters, S^0\setminus V(\mathcal M_A)\big)
&\eqBy{C\ref{cl:S0barS0iso}}\;
E_{\Gblack}\big(\XA\cap\bigcup\clusters, \bar S^0\setminus V(\mathcal M_A)\big)\\
\JUSTIFY{by~\eqref{eq:Hsplit}}&\;\subset\;
E_{\Gblack}\big(V_2(\mathcal M)\setminus V_2(\mathcal M_B),\bar S^0\setminus
V(\mathcal M_A)\big)\\
&~~~~\cup E_{\Gblack}\big(L\setminus (\smallatoms\cup V(\mathcal
M)),\bar S^0\setminus V(\mathcal M_A)\big)
\\
\JUSTIFY{by~\eqref{guanacos}, \eqref{eq:H0}}&\;\subset\;
E_{\Gblack}\big(V_2(\mathcal M)\setminus V_2(\mathcal M_B),\SR \setminus V(\M_A)\big)\;,
\end{align*}
as needed.
\end{proof}

In order to
prove~\eqref{newpropertyS6} we first observe that
\begin{align}
V(\mathcal N_1)\setminus V(\M_A\cup \M_B)&\eqByRef{eq:defMA} V(\mathcal
N_1)\setminus V\big( (\mathcal N_1\setminus \M)\cup \M_1\cup \M_B\big)\notag\\
&= (V(\mathcal N_1)\cap V(\mathcal \M))\setminus V(\M_B\cup \M_1)\notag\\
&\eqByRef{eq:SeparatedM} (V(\mathcal N_1)\cap V(\mathcal \M))\setminus V(\M')=  V(\M)\setminus V(\M')\;.\label{cl:matchINCL}
\end{align}

Now, we have
\begin{align*}
e_{\Gblack}(V(G)\setminus
V(\M_A\cup \M_B))&\le e_{\Gblack}(V(G)\setminus
V(\mathcal N_1))+\sum_{v\in V(\mathcal N_1)\setminus V(\M_A\cup
\M_B)}\deg_{\Gcapt}(v)\\
\JUSTIFY{by~\eqref{eq:bir1} and \eqref{cl:matchINCL}}&\le\sum_{v\in
V(\M)\setminus V(\M')}\deg_{\Gcapt}(v)\le |V(\M)\setminus V(\M')|\Omega^*k\\
\JUSTIFY{by~\eqref{eq:MsubMApp}} &<\epsilon
\Omega^* kn\;,
\end{align*}
which proves~\eqref{newpropertyS6}.
\smallskip

Let us turn to proving~\eqref{nicDoNAtom}. First, recall that we have
$V(\NAtom)\subset V(\M')\cup V(\mathcal N_1)$ (cf.~\ref{eq:defMA}). Since
$V(\mathcal N_1)\cap\smallatoms=\emptyset$ we actually have
\begin{equation}\label{eq:patchA}
V(\NAtom)=V(\NAtom)\cap V(\M')\;.
\end{equation}
Using~\eqref{eq:patchA} and~\eqref{theCandDwelike} we get
\begin{align}\nonumber
e_{\Gblack}\left(V(G)\setminus V(\mathcal
N_1),V(\NAtom)\right)&\le
e_{\Gblack}\left(V(G)\setminus V(\mathcal
N_1),V(\M')\cap \SR\right)\\
\nonumber
\JUSTIFY{by~\eqref{eq:645}}&\le 
e_{\Gblack}\left(V(G)\setminus V(\mathcal
N_1),(V(\M')\setminus V(\M))\cap \SR\right)\\
\nonumber
\JUSTIFY{by~\eqref{eq:M'Next}}&\le
e_{\Gblack}\left(V(G)\setminus V(\mathcal
N_1),(V(\M')\setminus V(\mathcal N_1))\cap \SR\right)\\
\label{eq:patchB}
&\le 2 e_{\Gblack}\left(V(G)\setminus V(\mathcal
N_1)\right)\eqByRef{eq:bir1}0\;. 
\end{align}
We have
\begin{align*}
e_{\Gblack}\left(V(G)\setminus V(\M_A\cup\M_B),V(\NAtom)\right) &\le e_{\Gblack}\left(V(G)\setminus V(\mathcal N_1),V(\NAtom)\right)\\ & \ \ +
e_{\Gblack}\left(V(\mathcal N_1)\setminus V(\M_A\cup\M_B),V(G)\right)\\
\JUSTIFY{by~\eqref{eq:patchB}}&\le 0+|V(\mathcal N_1)\setminus
V(\M_A\cup\M_B)|\Omega^* k\\
\JUSTIFY{by~\eqref{cl:matchINCL}, \eqref{eq:MsubMApp}}&\le \epsilon\Omega^* kn\;,
\end{align*}
as needed.

\smallskip

We have thus shown
Lemma~\ref{prop:LKSstruct}\eqref{prop6.1a}--\eqref{nicDoNAtom}. It only remains
to prove Lemma~\ref{prop:LKSstruct}\eqref{Mgoodisblack}, which we will do in the
remainder of this section.

\smallskip

We first collect several properties of $\XA$ and $\XC$.
The definitions of $\XC$ and $S^0$ give
\begin{equation}\label{eq:XCS0}
|\XC|(1+\eta)\frac{k}2 \le e_G\big(\XC,\SN \setminus V(\mathcal{M}_A\cup \mathcal{M}_B)\big) \le |\SN\setminus V(\mathcal{M}_A\cup \mathcal{M}_B)|(1+\eta)k\;.
\end{equation}

Each
vertex of $\XC$ has degree at least
$(1+\eta)\frac{k}2$ into $S$, and so,
\begin{equation}\label{eq:withrounding}
e_G(S,\XC)\ge |\XC|\left\lceil(1+\eta)\frac{k}2\right\rceil\;.
\end{equation}
Also, for each vertex $v\in \XC$, Definition~\ref{def:LKSsmall}\eqref{def:LKSsmallB} gives that
\begin{equation}\label{tomze}
\deg_G(v)=\lceil(1+\eta)k\rceil
\end{equation}
Consequently (using $\lceil a\rceil -\lceil\frac a2\rceil\le \frac a2$),
\begin{align}
e_G\big(\XA,\XC\big)
& \overset{\eqref{tomze}}\le  |\XC|\lceil(1+\eta)k\rceil - e_G(S,\XC)\notag \\
&\overset{\eqref{eq:withrounding}}\le  |\XC|(1+\eta)\frac{k}2\label{eq:PrI} \\
&\overset{\eqref{eq:XCS0}}\leq |\SN\setminus V(\mathcal{M}_A\cup
\mathcal{M}_B)|(1+\eta)k\;.\label{eq:eXAXC}
\end{align}

Let $\Mgood$ be defined as in Lemma~\ref{prop:LKSstruct}\eqref{Mgoodisblack},
that is, $\Mgood:=\{(X_1,X_2)\in \mathcal M_A\::\: X_1\cup X_2\subset \XA \}$. Note that~\eqref{eq:Asub} implies that $X_1\subseteq S$
for every $(X_1,X_2)\in \mathcal M_B$. Thus by the definition of $\XA$, 
\begin{equation}\label{cl:som}
\text{if $(X_1,X_2)\in \mathcal M_A\cup\mathcal M_B$ with $X_1\cup X_2\subset L$, then 
$(X_1,X_2)\in \Mgood$.}
\end{equation}

We will now prove the first part of Lemma~\ref{prop:LKSstruct}\eqref{Mgoodisblack}, that is, we show that each $\Mgood$-edge is an edge
of $\BGblack$. Indeed,
by~\eqref{theCandDwelike}, we have that $V_1(\mathcal
M_1)\subset S$, so as $\XA\cap S=\emptyset$, it
follows that $\mathcal M_1\cap \Mgood=\emptyset$. Thus  $\Mgood\subset \mathcal N_1$. As $\mathcal N_1$ corresponds to a matching in $\BGblack$, all is as desired.

Finally, let us assume that
neither~{\bf(K1)} nor~{\bf(K2)} is fulfilled. After five preliminary
observations (Claim~\ref{claim:DZ4}--Claim~\ref{claim:DZ5}), we will derive a
contradiction from this assumption.

\begin{claim}\label{claim:DZ4}
We have
$|S\cap V(\mathcal M_A)|\le|\XA\cap V(\mathcal M_A)|$.
\end{claim}
\begin{proof}[Proof of Claim~\ref{claim:DZ4}]
To see this, recall that each $\mathcal M_A$-vertex $X\in\V(\mathcal M_A)$ is
either contained in $S$, or in $L$. Further, if $X\subset S$ then
its partner in $\mathcal M_A$ must be in $L$, as $S$ is
independent. Now, the claim follows after noticing that $L\cap
V(\mathcal M_A)=\XA\cap V(\mathcal M_A)$.
\end{proof}

\begin{claim}\label{claim:DZ1}
We have $|S\setminus V(\mathcal M_A\cup \mathcal M_B)|+2\eta
n< |\XA\setminus V(\mathcal M_A)|+\eta n/3$.
\end{claim}
\begin{proof}[Proof of Claim~\ref{claim:DZ1}]
As $G\in\LKSgraphs{n}{k}{\eta}$, we have $|S|+2\eta n\le
|L|$. Therefore,
\begin{align*}
|S\setminus V(\mathcal M_A\cup \mathcal M_B)|+2\eta
n\le & \ |L\setminus V(\mathcal M_A\cup \mathcal
M_B)|+\sum_{\substack{(X_1,X_2)\in\mathcal
M_A\cup\mathcal M_B\\X_1\cup X_2\subset L}}|X_1\cup X_2|\\
\eqBy{\eqref{cl:som}} & |\XA\setminus V(\mathcal M_A)|+|V(\Mgood)|\\
\lBy{$\neg${\bf(K2)}} & |\XA\setminus V(\mathcal M_A)|+\eta n/3\;.
\end{align*}
\end{proof}

\begin{claim}\label{claim:DZ2}
We have $e_{\Gcapt}\left(\XA\cap(\smallatoms\cup V(\mathcal M))
,
\SR\setminus V(\mathcal M_A)
\right)<\eta kn/2$.
\end{claim}
\begin{proof}[Proof of Claim~\ref{claim:DZ2}]
As 
\begin{align*}
\XA\cap (\smallatoms\cup V(\mathcal M))&\subset \left((L\cap \smallatoms)\cup
V_2(\mathcal M)\right)\setminus V_2(\mathcal M_B)\quad \mbox{and}\\
\SR\setminus V(\mathcal M_A)
&\subset
\left(\left(\SR\setminus V(\mathcal N_1)\right)\cup V_1(\mathcal
M)\right)\setminus V_1(\mathcal M_1)\;,
 \end{align*}
 we get from~\eqref{eq:SeparatedM} that
\begin{equation}\label{eq:iiModif1}
e_{\GD}\left(\XA\cap(\smallatoms\cup V(\mathcal M))
,
\SR\setminus V(\mathcal M_A)
\right)\le \gamma kn\;.
\end{equation}
Observe now that both sets $\XA\cap (\smallatoms\cup V(\mathcal M))$ and
$\SR\setminus V(\mathcal M_A)$ avoid $\HugeVertices$. Further, no edges between them belong to $\Gexp$, because Claim~\ref{cl:SRS0} implies that $\SR\setminus V(\mathcal
M_A)\subset S^0\subseteq V(G)\setminus V(\Gexp)$. Therefore, we can pass from $\GD$ to $\Gcapt$ in~\eqref{eq:iiModif1} to get
\begin{equation*}
e_{\Gcapt}\left(\XA\cap(\smallatoms\cup V(\mathcal M))
,
\SR\setminus V(\mathcal M_A)
\right)\le \gamma kn<\eta kn/2\;.
\end{equation*}
\end{proof}

\begin{claim}\label{claim:DZ3}
We have
$S\setminus (\SR\cup V(\mathcal M_A) )
\subset
S\setminus (\bar S^0\cup V(\mathcal M_A\cup\mathcal M_B))$.
\end{claim}
\begin{proof}[Proof of Claim~\ref{claim:DZ3}]
The claim follows directly from the following two inclusions.
\begin{align}
\label{eq:FC1}
\SR\cup V(\mathcal M_A)&\supset S\cap V(\mathcal M_A\cup\mathcal
M_B)\;\mbox{, and}\\
\label{eq:FC2}
\SR\cup V(\mathcal M_A)&\supset \bar S^0\;.
\end{align}
Now,~\eqref{eq:FC1} is trivial, as by~\eqref{theCandDwelike} we have that
$\SR\supset S\cap V(\mathcal M_B)$.
 To see~\eqref{eq:FC2}, it suffices
by~\eqref{eq:defMA} to prove that $V(N_1\setminus M)\cup\BSR\supset \BSN$.
This
is however the assertion of Claim~\ref{cl:S0mN1}.
\end{proof}

Next, we bound $e_{\Gcapt}\big(\XA,S\big)$. 
\begin{claim}\label{claim:DZ5}
We have $$e_{\Gcapt}\big(\XA,S\big)\le |S\cap V(\mathcal M_A)|(1+\eta)k+|S\setminus(\SN\cup V(\mathcal M_A\cup \mathcal M_B))|(1+\eta)k 
+\frac12\eta kn\;.$$
\end{claim}
\begin{proof}[Proof of Claim~\ref{claim:DZ5}]
We have 
\begin{align*}
e_{\Gcapt}\big(\XA,S \big)
=& \ e_{\Gcapt}\big(\XA,S\cap V(\mathcal M_A)\big)
+  e_{\Gcapt}\big(\XA,S\setminus(\SR\cup V(\mathcal M_A))\big) 
\\ 
&\mbox{~~}
+e_{\Gcapt}\big(\XA\setminus (\smallatoms\cup V(\mathcal M)),\SR\setminus V(\mathcal M_A)\big)
+e_{\Gcapt}\left(\XA\cap(\smallatoms\cup V(\mathcal M)),\SR\setminus V(\mathcal
M_A)\right)\;.
\end{align*}
To bound the first term we use that each vertex in $S\cap V(\mathcal M_A)$ has degree at most $(1+\eta) k$, and thus obtain $e_{\Gcapt}(\XA,S\cap V(\mathcal M_A))\le |S\cap V(\mathcal M_A)|(1+\eta)k$. To bound the second term, we again use a bound on degree of vertices of $S\setminus\big((\SR\cup V(\mathcal
M_A))\cup(S^0\setminus \bar S^0))$, together with Claim~\ref{claim:DZ3}. The third term is zero by Claim~\ref{cl:CrossEdges1}. The fourth term can be bounded by Claim~\ref{claim:DZ2}.
\addtocounter{theorem}{1}
\end{proof}

A relatively short double counting below will lead to the final contradiction. The idea behind this computation is given in Figure~\ref{fig:strukturaContradiction}.
\begin{figure}[t] \centering
\includegraphics[scale=0.9]{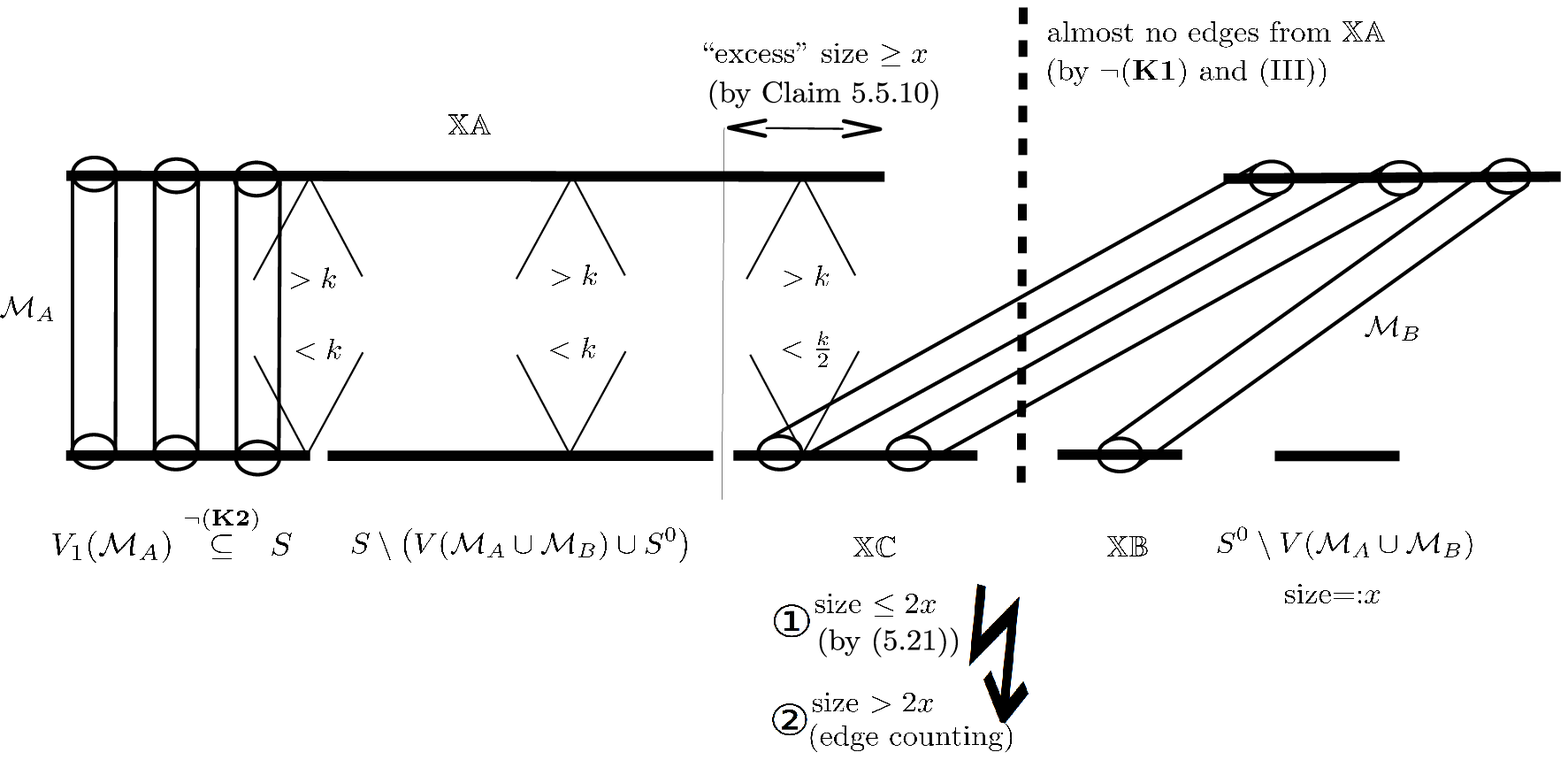}
\caption[Contradiction in Lemma~\ref{prop:LKSstruct}]{A simplified computation showing that $\neg {\bf(K1)}$, $\neg {\bf(K2)}$ leads to a contradiction. Denoting by $x$ the size of $S^0\setminus V(\M_A\cup\M_B)$ we get \ding{172} $|\XC|\le 2x$. On the other hand, each vertex of $\XA$ emanates $\gtrsim k$ edges which are absorbed by the sets $V_1(\M_A)$, $S\setminus (V(\M_A\cup\M_B)\cup S^0)$, and $\XC$. The vertices of $V_1(\M_A)$ and $S\setminus (V(\M_A\cup\M_B)\cup S^0)$ can absorb $\lesssim k$ edges. The vertices of $\XC$ receive $\lesssim\frac k2$ edges of $\XA$ by~\eqref{eq:PrI}. This leads to \ding{173} $|\XC|> 2x$, doubling the size of the ``excess'' vertices of $\XA$.} 
\label{fig:strukturaContradiction}
\end{figure}
\begin{align}
\begin{split}\label{eq:MCS}
|\XA|(1+\eta)k&\le \sum_{v\in \XA}\deg_G(v) \le \sum_{v\in \XA}\deg_{\Gcapt}(v)+2\big(e(G)-e(\Gcapt)\big)\\
& \le
2e_{\Gcapt}(\XA)+e_{\Gcapt}(\XA,\XB)+e_{\Gcapt}\big(\XA,\XC\big)+ e_{\Gcapt}\big(\XA,S\big)+\frac{\eta
kn}3\\[6pt]
\JUSTIFY{by $\neg${\bf(K1)}, \eqref{eq:eXAXC}, C\ref{claim:DZ5}}
&\le \frac 76\eta kn+ \big|\SN\setminus V(\mathcal{M}_A\cup
\mathcal{M}_B)\big|(1+\eta)k\\
&\mbox{~~~~~~}+ |S\cap V(\mathcal M_A)|(1+\eta)k\\
&\mbox{~~~~~~}+    |S\setminus(\SN\cup V(\mathcal M_A\cup \mathcal M_B))|(1+\eta)k
\\[6pt]
\JUSTIFY{by C\ref{claim:DZ4}}
&\le \frac 76\eta kn+|S\setminus V(\mathcal M_A\cup\mathcal
M_B)|(1+\eta)k\\
&\mbox{~~~~~~}+|\XA\cap V(\mathcal M_A)|(1+\eta)k\\[6pt]
\JUSTIFY{by C\ref{claim:DZ1}}
&\le \frac 76 \eta kn+\big(|\XA\setminus V(\mathcal M_A)|-\frac53\eta
n\big)(1+\eta)k\\
 &\mbox{~~~~~~}+|\XA\cap V(\mathcal M_A)|(1+\eta)k
\\[6pt]
&< |\XA|(1+\eta)k-\frac12\eta kn\;,
\end{split}
\end{align}
a contradiction. This completes the proof of Lemma~\ref{prop:LKSstruct}.
\end{proof}

\section{Acknowledgements}\label{sec:ACKN}
The work on this project lasted from the beginning of 2008 until 2014
and we are very grateful to the following institutions and funding bodies for
their support. 

\smallskip

During the work on this paper Hladk\'y was also affiliated with Zentrum
Mathematik, TU Munich and Department of Computer Science, University of Warwick. Hladk\'y was funded by a BAYHOST fellowship, a DAAD fellowship, 
Charles University grant GAUK~202-10/258009, EPSRC award EP/D063191/1, and by an EPSRC Postdoctoral Fellowship during the work on the project. 

Koml\'os and Szemer\'edi acknowledge the support of NSF grant
DMS-0902241.

Piguet has been also affiliated with the Institute of Theoretical Computer Science, Charles University in Prague, Zentrum
Mathematik, TU Munich, the Department of Computer Science and DIMAP,
University of Warwick, and the School of Mathematics, University of Birmingham. Piguet acknowledges the support of the Marie Curie fellowship FIST,
DFG grant TA 309/2-1, a DAAD fellowship,
Czech Ministry of
Education project 1M0545,  EPSRC award EP/D063191/1,
and  the support of the EPSRC
Additional Sponsorship, with a grant reference of EP/J501414/1 which facilitated her to
travel with her young child and so she could continue to collaborate closely
with her coauthors on this project. This grant was also used to host Stein in
Birmingham.  Piguet was supported by the European Regional Development Fund (ERDF), project ``NTIS --- New Technologies for Information Society'', European Centre of Excellence, CZ.1.05/1.1.00/02.0090.

Stein was affiliated with the Institute of Mathematics and Statistics, University of S\~ao Paulo, the Centre for Mathematical Modeling, University of Chile and the Department of Mathematical Engineering, University of Chile. She was
supported by a FAPESP fellowship, and by FAPESP travel grant  PQ-EX 2008/50338-0, also
CMM-Basal,  FONDECYT grants 11090141 and  1140766. She also received funding by EPSRC Additional Sponsorship EP/J501414/1.

We enjoyed the hospitality of the School of Mathematics of University of Birmingham, Center for Mathematical Modeling, University of Chile, Alfr\'ed R\'enyi Institute of Mathematics of the Hungarian Academy of Sciences and Charles University, Prague, during our long term visits.

The yet unpublished work of Ajtai, Koml\'os, Simonovits, and Szemer\'edi on the Erd\H{o}s--S\'os Conjecture was the starting point for our project, and our solution crucially relies on the methods developed for the Erd\H{o}s-S\'os Conjecture. Hladk\'y, Piguet, and Stein are very grateful to the former group for explaining them those techniques.

\medskip
A doctoral thesis entitled \emph{Structural graph theory} submitted by Hladk\'y in September 2012 under the supervision of Daniel Kr\'al at~Charles University in~Prague is based on the series of the papers~\cite{cite:LKS-cut0,cite:LKS-cut1, cite:LKS-cut2, cite:LKS-cut3}. The texts of the two works overlap greatly. We are grateful to PhD committee members Peter Keevash and Michael Krivelevich. Their valuable comments are reflected in the series. 

\bigskip
We thank the referees for their very detailed remarks.

\bigskip
The contents of this publication reflects only the authors' views and not necessarily the views of the European Commission of the European Union.

\printindex{mathsymbols}{Symbol index}
\printindex{general}{General index}

\newpage
\addcontentsline{toc}{section}{Bibliography}
\bibliographystyle{alpha}
\bibliography{bibl}
\end{document}